\documentclass[final]{amsart}
\usepackage[pdfencoding=auto, psdextra]{hyperref}
\usepackage{amssymb, amsthm, mathrsfs}

\usepackage{biblatex}
\theoremstyle{plain}
\newtheorem*{theorem*}{Theorem}

\newtheorem{corollary}{Corollary}[section]
\newtheorem{lemma}[corollary]{Lemma}
\newtheorem{proposition}[corollary]{Proposition}
\newtheorem{theorem}[corollary]{Theorem}

\theoremstyle{definition}

\newtheorem{example}[corollary]{Example}
\newtheorem{remark}[corollary]{Remark}

\usepackage{tikz}
\usetikzlibrary{cd}
\usetikzlibrary{decorations.markings}

\usepackage{enumitem}

\newcommand{\qbinom}{\genfrac{[}{]}{0pt}{}}

\newcommand{\calB}{\mathcal B}

\newcommand{\calM}{\mathcal M}

\newcommand{\bbN}{\mathbb N}
\newcommand{\bbQ}{\mathbb Q}
\newcommand{\bbR}{\mathbb R}
\newcommand{\bbZ}{\mathbb Z}

\newcommand{\frakg}{\mathfrak g}
\newcommand{\frakh}{\mathfrak h}

\newcommand{\ula}{{\underline a}}
\newcommand{\ulb}{{\underline b}}
\newcommand{\uli}{{\underline i}}
\newcommand{\ulj}{{\underline j}}
\newcommand{\ulk}{{\underline k}}
\newcommand{\ull}{{\underline l}}
\newcommand{\ulm}{{\underline m}}

\newcommand{\te}{\tilde e}
\newcommand{\tf}{\tilde f}

\DeclareMathOperator{\id}{id}

\DeclareMathOperator{\trace}{tr}
\DeclareMathOperator{\weight}{wt}

\DeclareMathOperator{\ad}{ad}

\DeclareMathOperator{\Fr}{Fr}
\DeclareMathOperator{\Hom}{Hom}

\newcommand{\drawDoubleEdge}[2]{%
    \draw([yshift=1.5pt]#1.east) -- ([yshift=1.5pt]#2.west);
    \draw ([yshift=-1.5pt]#1.east) -- ([yshift=-1.5pt]#2.west);
    \draw[decoration={markings,mark=at position 0.7 with \arrow[opaque]{Classical TikZ Rightarrow[length=4pt]}}, postaction={decorate}, transparent] (#1) -- (#2);
}

\newcommand{\drawReverseDoubleEdge}[2]{%
    \draw([yshift=1.5pt]#1.east) -- ([yshift=1.5pt]#2.west);
    \draw ([yshift=-1.5pt]#1.east) -- ([yshift=-1.5pt]#2.west);
    \draw[decoration={markings,mark=at position 0.4 with \arrowreversed[opaque]{Classical TikZ Rightarrow[length=4pt]}}, postaction={decorate}, transparent] (#1) -- (#2);
}

\newcommand{\drawTripleEdge}[2]{%
    \draw([yshift=2pt]#1.east) -- ([yshift=2pt]#2.west);
    \draw ([yshift=-2pt]#1.east) -- ([yshift=-2pt]#2.west);
    \draw[decoration={markings,mark=at position 0.3 with \arrowreversed{Classical TikZ Rightarrow[length=5pt]}}, postaction={decorate}] (#1) -- (#2);
}

\title[Imaginary vectors, tight monomial cones, quantum Frobenius]{Dual imaginary vectors, tight monomial cones and quantum Frobenius morphism}
\author{Felix Röhrich}
\address{Lehrstuhl für Algebra und Darstellungstheorie, RWTH Aachen University, Pontdriesch 10-16, 52062 Aachen, Germany}
\email{roehrich.math@pm.me}

\begin{document}

\begin{abstract}
    We show that the quantum Frobenius morphism and its splitting are not fully compatible with the canonical basis for any finite-dimensional simple Lie algebra if the rank is sufficiently large.
    The incompatibility occurs at same place where Leclerc found his imaginary vectors, and where there are monomials in the tight monomial cone which do not belong to the canonical basis.
\end{abstract}

\maketitle

\section{Introduction}

Let $\frakg$ be a (finite-dimensional) simple Lie algebra and $U_q(\frakg)$ its quantized enveloping algebra.
The canonical basis, introduced by Lusztig \cite{Lusztig:CanonicalBases-QuantizedEnvelopingAlgebras,Lusztig:Quivers-PerverseSheaves-QuantizedEnvelopingAlgebras} and Kashiwara \cite{Kashiwara:CrystalBases-QuantizedUniversalEnvelopingAlgebras}, is a distinguished basis of $U_q^-(\frakg)$ with remarkable properties.
It is relatively well-understood in the types $A_1$, $A_2$, $A_3$ and $B_2$ due to explicit descriptions available there \cite{Lusztig:CanonicalBases-QuantizedEnvelopingAlgebras,Xi:CanonicalBasis-A3,Xi:CanonicalBasis-B2}.
For $A_4$ an effort has been made to obtain a complete description \cite{Hu-Ye-Yue:CanonicalBasis-A4-I-Monomial-Elements,Hu-Ye:CanonicalBasis-A4-II-PolynomialElements-One-Variable}, which presents a similar picture as in the previous cases.
However, a fundamental change occurs in $A_5$; this can be seen from the following examples:

The singular support of canonical basis elements is not necessarily irreducible \cite{KashiwaraSaito:GeometricConstruction-CrystalBases};
there are imaginary vectors in the dual canonical basis \cite{Leclerc:ImaginaryVectors-DualCanonicalBasis};
the Frobenius morphism and its splitting are not fully compatible with the canonical basis \cite{Baumann:CanonicalBasis-QuantumFrobeniusMorphism};
there are monomials in Lusztig's tight monomial cone which do not belong to the canonical basis;
there is a counterexample to an analogue of the James conjecture for Khonaov-Lauda-Rouquier algebras \cite{Williamson:Analogue-JamesConjecture};
and the cluster algebra structure on quantum coordinate ring $A_q(\frakg; w_0)$ has infinite cluster type \cite{Geiss-Leclerc-Schroer:SemicanonicalBases-PreprojectiveAlgebras}.

In this paper we will focus on the tight monomial cone, the duals of Leclerc's imaginary vectors, and the quantum Frobenius morphism and its splitting.

The \emph{tight monomial} was found by Lusztig, when he investigated under which conditions a monomial of $U_q^-(\frakg)$ belongs to the canonical basis \cite{Lusztig:TightMonomials-QuantizedEnvelopingAlgebras}.
He called a monomial tight if it is an element of the canonical basis, and semi-tight if it is an $\bbN$-linear combination of elements of the canonical basis.
His examples showed that in $A_2$ and $A_3$ a monomial is tight if its exponent vector belongs to the tight monomial cone.
He thus posed the question if this holds true in general.
This turned out to be true in $A_4$ \cite{Marsh:TightMonomials-QuantizedEnvelopingAlgebras} and $B_2$ \cite{Xi:CanonicalBasis-B2}, however Reineke \cite{Reineke:Monomials-CanonicalBases-QuantumGroups-QuadraticForms} showed for $A_6$ that there is, in fact, a 6-dimensional region for which this is false.
Despite the initial strong interest in the tight monomial cone, it has become neglected after the inception of cluster algebras.
Only sporadic results (e.g. \cite{CalderoMarshMorier-Genoud:Realisation-LusztigCones,FangFourierReineke:Cones-QuantumGroups-TropicalFlagVarieties}) have been obtained since then, nevertheless they highlight that the tight monomial deserves further investigation.

The \emph{imaginary vectors} are special elements of the dual canonical basis.
They were introduced by Leclerc \cite{Leclerc:ImaginaryVectors-DualCanonicalBasis} as counterexamples to a conjecture Berenstein and Zelevinsky \cite{BerensteinZelevinsky:StringBases-QuantumGroups-Ar} on the multiplicative properties of the dual canonical basis.
They conjectured that for any two elements $b_1$ and $b_2$ of the dual canonical basis, their product $b_1b_2$ belongs again to the canonical basis up to a power of $q$ if and only if $b_1$ and $b_2$ $q$-commute; which is in fact true in $A_n$ up to $n = 4$ and $B_2$.
The imaginary vectors have also found application in the representation theory of quantum affine algebras \cite{BritoChari:HigherOrderKirillovReshetikhinModules-ImaginaryModules-MonoidalCategorification}.

The \emph{quantum Frobenius morphism} $\Fr$ \emph{and its splitting} $\Fr'$ play an important role in representation theory and geometry; for example:
$\Fr$ is the obvious quantum analogue of the Frobenius morphism for reductive groups over a field of finite characteristic, which plays an important role in their representation theory.
Kumar and Littelmann \cite{KumarLittelmann:FrobeniusSplitting-CharacteristicZero-QuantumFrobenius} showed that the dual of $\Fr$ can be viewed as characteristic zero analogue of the geometric Frobenius splitting.
The path vectors were constructed by Littelmann \cite{Littelmann:ContractingModules-StandardMonomialTheory} using the Frobenius splitting $\Fr'$.
Song \cite{Song:QuantumFrobeniusSplitting-ClusterStructures} proved, under certain assumptions, that duals of the quantum Frobenius morphism and its splitting are compatible with the cluster monomials.

The main result of this paper is the following theorem.

\begin{theorem*}
    The quantum Frobenius morphism and its splitting are not fully compatible with the canonical basis in the types $A_n$ ($n \ge 5$), $B_n$ ($n \ge 3$), $C_n$ ($n \ge 3$), $D_n$ ($n \ge 4$), $E_6$, $E_7$, $E_8$, $F_4$ and $G_2$.
    The incompatibility occurs at same place where Leclerc found his imaginary vectors, and where there are monomials in the tight monomial cone which do not belong to the canonical basis.
\end{theorem*}

The result on the quantum Frobenius morphism and its splitting for $A_5$ and $D_4$ was already obtained by Baumann \cite{Baumann:CanonicalBasis-QuantumFrobeniusMorphism}.
He also noted the connection to Leclerc's imaginary vectors.
Our contribution is extending the result to all finite types and establishing the connection to the tight monomial cone.

The paper is structured as follows.
In Section \ref{section:preliminaries} we introduce the notation and recall the definition and important properties of the canonical basis from Kashiwara's point of view.

In Section \ref{section:tight-monomial-cones} we recall the definition of the tight monomial cone.
Then we show the it embeds into the adapted string cone defined by Littelmann \cite{Littelmann:Cones-Crystals-Patterns}, which allows us to give a sufficient condition for a monomial in the tight monomial cone to belong to the canonical basis (Proposition \ref{proposition:lusztig-cone-canonical-basis}).
Using this condition we find a family of monomials which belong to the canonical basis (Proposition \ref{proposition:family-of-monomials-in-canonical-basis}).

In Section \ref{section:dual-imaginary-vectors} we then use the tight monomial cone to study the duals of Leclerc's imaginary vectors and their squares (up to normalization).
Using this perspective we identify these elements as part of a larger family.
Surprisingly, there is only a finite number of pairs $\theta^{[l]}$ and $\xi^{[l]}$ having these properties (Theorem \ref{theorem:classification}).
Following this, we compute the values of the bilinear form on $U_q^-(\frakg)$ for all pairings of $\theta^{[l]}$ and $\xi^{[l]}$ (Theorem \ref{theorem:bilinear-form-values}).
Finally, we apply these results give examples of monomials in the tight monomial cone which do not belong to the canonical basis (Section \ref{section:application-tight-monomial-cones}).

In Section \ref{section:quantum-frobenius} we use the explicit description of the canonical basis to show that the quantum Frobenius morphism and its splitting are fully compatible with the canonical basis in the types $A_1$, $A_2$, $A_3$ and $B_2$.
Afterwards, we provide examples in types $A_5$, $B_3$, $C_3$, $D_4$, and $G_2$ where both maps are not fully compatible with the canonical basis.

\subsection*{Acknowledgements}
I want to thank Xin Fang for helpful discussions and comments on the manuscript, and Markus Reineke for his interest in this project and suggestions.
I acknowledge funding by the Deutsche Forschungsgemeinschaft (DFG, German Research Foundation) - project number 442047500 - through the Collaborative Research Center ``Sparsity and Singular Structures'' (SFB 1481).

\section{Preliminaries}
\label{section:preliminaries}

\subsection{Quantized enveloping algebras}

\subsubsection*{Root Datum}

Throughout this paper we fix a \emph{root datum} consisting of
\begin{enumerate}
	\item A finite dimensional $\bbQ$-vector space $\frakh$,
	\item a finite index set $I$,
	\item a linearly independent set $\{ \alpha_i \mid i \in I \}$ of $\frakh^*$ (the set of \emph{simple roots}),
	\item a subset $\{ h_i \mid i\in I\}$ of $\frakh$ (the set of \emph{simple coroots}),
	\item a $\bbQ$-valued symmetric bilinear form $(\cdot, \cdot)$ on $\frakh^*$, and
	\item a lattice $P \subseteq \frakh^*$ (the \emph{weight lattice}) together with the natural pairing $\langle \cdot, \cdot\rangle \colon P^\vee \otimes P \to \bbZ$, where $P^\vee := \Hom_\bbZ(P, \bbZ)$ is the dual lattice,
\end{enumerate}
such that for $i\in I$ the following properties are fulfilled:
\begin{enumerate}
	\item $(\langle h_i, \alpha_j\rangle)_{ij}$ is a Cartan matrix,
	\item $(\alpha_i, \alpha_i) \in 2\bbZ_{>0}$,
	\item $\langle h_i, \lambda\rangle = 2(\alpha_i, \lambda)/(\alpha_i, \alpha_i)$ for $\lambda \in P$,
	\item $\alpha_i\in P$ and $h_i\in P^\vee$.
\end{enumerate}

We assume that the root datum is \emph{simply-connected}, that is, there exist $\varpi_i \in P$ such that $\langle h_i, \varpi_j\rangle = \delta_{ij}$ for any $i, j\in I$.

Let $Q = \bigoplus_{i\in I}\bbZ\alpha_i \subseteq P$ be the \emph{root lattice}; we set $Q_+ := \bigoplus_{i\in I} \bbZ_{\ge0} \alpha_i$ and $Q_- := -Q_+$.
For $\xi = \sum_{i\in I} \xi_i\alpha_i \in Q$, define $\trace(\xi) = \sum_{i\in I} \xi_i$.
We call a weight $\lambda \in P$ \emph{dominant} if $\langle h_i, \lambda\rangle \ge 0$ for all $i\in I$, and denote set of dominant weights by $P_+$.
A weight $\lambda$ is \emph{regular} if $\langle h_i, \lambda\rangle \neq 0$ for all $i\in I$.
We define the \emph{Weyl vector} $\rho := \sum_{i\in I} \varpi_i$.

The weight lattice $P$ has a partial order given by $\lambda \succeq \mu$ if $\lambda - \mu \in Q_+$, that is, $\lambda - \mu$ is a $\bbZ_{\ge0}$-linear combination of simple roots.

\subsubsection*{Kac-Moody Lie algebra}
\label{section:definition-kac-moody-lie-algebra}

Let $\frakg$ be the \emph{symmetrizable Kac-Moody Lie algebra} associated with the Cartan matrix $(\langle h_i, \alpha_j\rangle)_{ij}$.
That is to say, $\frakg$ is the $\bbQ$-vector space generated by $e_i$, $f_i$ ($i \in I$) and $\frakh$ with the Lie bracket
\[
    [h, h'] = 0,\quad [h, e_i] = \langle h, \alpha_i\rangle e_i,\quad [h, f_i] = -\langle h, \alpha_i\rangle f_i, \quad [e_i, f_j] = \delta_{ij} h_i,
\]
for $i, j \in I$ and $h, h' \in \frakh$, and the relations for $i \neq j$
\[
    \ad(e_i)^{1-\langle h_i, \alpha_j\rangle} e_j = 0, \quad \ad(f_i)^{1-\langle h_i, \alpha_j\rangle} f_j = 0.
\]

\subsubsection*{Weyl Group}
\label{section:definition-weyl-group}

For $i\in I$ define the \emph{simple reflection} $s_i\colon \frakh^* \to \frakh^*$,
\begin{equation}
    \label{equation:simple-reflection}
    s_i \lambda = \lambda - \langle h_i, \lambda\rangle \alpha_i.
\end{equation}
The group $W$ generated by the $s_i$ ($i\in I$) is called the \emph{Weyl group}.

For a finite sequence $\uli = (i_1, \dots, i_r)$ in $I$ we write
\begin{equation}
    s_\uli = s_{i_1} \cdots s_{i_r}.
\end{equation}

Given $w \in W$, its \emph{length} $\ell(w)$ is the minimal integer $r$ such that $w = s_\uli$ for some sequence $\uli = (i_1, \dots, i_r)$ in $I$; such a word $s_\uli = s_{i_1}\cdots s_{i_r}$ is called a \emph{reduced decomposition} of $w$.
Similarly, we say that a sequence $\uli$ in $I$ is \emph{reduced} if $s_\uli$ is a reduced decomposition of some element of $W$.

\subsubsection*{Quantized enveloping algebra}
\label{section:definition-quantized-enveloping-algebra}

The \emph{quantized enveloping algebra} $U_q(\frakg)$ is the algebra over the rational function field $\bbQ(q)$ generated by $e_i, f_i$ ($i\in I$) and $q^h$ ($h\in P^\vee$) subject to the following relations for $i, j \in I$ and $h, h' \in P^\vee$:
\begin{enumerate}
    \item $q^0 = 1$, $q^h q^h = q^{h+h'}$,
    \item $q^h e_i q^{-h} = q^{\langle h, \alpha_i\rangle} e_i$, $q^h f_i q^{-h} = q^{-\langle h, \alpha_i\rangle} f_i$
    \item $[e_i, f_j] = \delta_{ij} (t_i - t_i^{-1})/(q_i - q_i^{-1})$, and
    \item $\sum_{k=0}^a (-1)^k e_i^{(k)} e_j e_i^{(a-k)} = \sum_{k=0}^a (-1)^k f_i^{(k)} f_j f_i^{(a-k)} = 0$ for $i \neq j$, $a = 1 - \langle h_i, \alpha_j\rangle$.
\end{enumerate}
Here we set $q_i = q^{(\alpha_i,\alpha_i)/2}$, $t_i = q_i^{h_i}$; $e_i^{(n)} = e_i^n/[n]_i!$ and $f_i^{(n)} = f_i^n/[n]_i!$ are the \emph{divided powers} as usual, where
\[
	[n]_i := \frac{q_i^n - q_i^{-n}}{q_i - q_i^{-1}} \quad \text{for } n \in \bbZ \quad \text{and} \quad [n]_i! := [n]_i [n-1]_i \cdots [1]_i \quad \text{for } n \in \bbZ_{\ge0}.
\]
Furthermore, we put
\[
    \qbinom{n}{k} := \frac{[n] [n-1]_q\cdots [n-k+1]}{[k] [k-1] \cdots [1]} \quad \text{for } n \in \bbZ \text{ and } k \in \bbN.
\]
We will also use the notations
\begin{equation}
    \qbinom{t_i; n}{k} := \prod_{j=1}^k \frac{q_i^{n+1-j} t_i - q_i^{-(n+1-j)} t_i^{-1}}{q_i^j - q_i^{-j}}
\end{equation}
for $n\in \bbZ$ and $k\in \bbZ_{\ge0}$, and
\[
    \qbinom{t_i}{k} := \qbinom{t_i; 0}{k}.
\]

For $n\in \bbZ$ we use the convention that $f_i^{(n)} = e_i^{(n)} = 0$ if $n < 0$.

For a sequence $\uli = (i_1, \dots, i_r)$ in $I$ and $\ula = (a_1, \dots, a_r) \in \bbZ^r$, we set
\begin{equation}
    f_\uli^{(\ula)} := f_{i_1}^{(a_1)} \cdots f_{i_r}^{(a_r)} \in U_q^-(\frakg).
\end{equation}

\subsubsection*{Weight grading}
The algebra $U_q(\frakg)$ has a natural $Q$-grading; for $\xi \in Q$ the graded component is
\[
    U_q(\frakg)_\xi = \{ u \in U_q(\frakg) \mid q^h u q^{-h} = q^{\langle h, \xi\rangle} u \text{ for all } h \in P^\vee \}.
\]
If $u \in U_q(\frakg)_\xi$, we say that $\xi$ is the \emph{weight} of $u$ and denote it by $\weight u$.

\subsubsection*{Automorphisms}

The \emph{bar involution} is the $\bbQ$-algebra automorphism $-\colon U_q(\frakg) \to U_q(\frakg)$ given by
\begin{equation}
    \overline{e_i} = e_i,\quad \overline{f_i} = f_i,\quad \overline{q}= q^{-1},\quad \overline{q^h} = q^{-h}.
\end{equation}

The \emph{$*$-involution} is the $\bbQ(q)$-algebra antiautomorphism $*\colon U_q(\frakg) \to U_q(\frakg)$ given by
\begin{equation}
    *e_i = e_i,\quad *f_i = f_i,\quad *q^h = q^{-h}.
\end{equation}

\subsection{Canonical basis}

We recall the definition of the canonical basis of $U_q^-(\frakg)$ following Kashiwara \cite{Kashiwara:CrystalBases-QuantizedUniversalEnvelopingAlgebras}.

\subsubsection*{$\bbZ$-forms}

Let $U_\bbZ(\frakg)$ be the $\bbZ[q, q^{-1}]$-subalgebra of $U_q(\frakg)$ generated by $e_i^{(n)}$, $f_i^{(n)}$, $t_i$, and $\qbinom{t_i; a}{n}$ ($i\in I$, $a\in \bbZ$, $n \ge 0$) .
Let $U_\bbZ^-(\frakg)$ be the $\bbZ[q,q^{-1}]$-subalgebra of $U_q^-(\frakg)$ generated by $f_i^{(n)}$.

\subsubsection*{Bilinear form}

For $i\in I$, let $e_i'\colon U_q^-(\frakg) \to U_q^-(\frakg)$ be the skew-derivation given by
\[
    e_i'f_j = q_i^{-\langle h_i, \alpha_j\rangle} f_j e_i' + \delta_{ij}.
\]
Here $f_j$ acts on $U_q^-(\frakg)$ by left multiplication.

We will use the following formulas in our computations without mention.
\begin{lemma}[\cite{Kashiwara:CrystalBases-QuantizedUniversalEnvelopingAlgebras}]
    \label{lemma:q-boson-commutation-formulas}
    For $n, m \ge 0$, we have
	\[
		e_i'^n f_i^{(m)} = \sum_{k=0}^n q_i^{-2nm+k(n+m)-k(k-1)/2} \qbinom{n}{k}_i f_i^{(m-k)} e_i'^{n-k}
	\]
	and for $j \neq i$
	\[
		e_i'^n f_j^{(m)} = q_i^{-\langle h_i, \alpha_j\rangle nm} f_j^{(m)}e_i'^n.
	\]
\end{lemma}

By \cite{Kashiwara:CrystalBases-QuantizedUniversalEnvelopingAlgebras} there is a unique non-degenerate symmetric bilinear form $(\cdot, \cdot)$ on $U_q^-(\frakg)$ with values in $\bbQ$ such that
\begin{equation}
    (1, 1) = 1 \text{ and } (f_iu, v) = (u, e_i'v) \text{ for } u, v\in U_q^-(\frakg).
\end{equation}

\subsubsection*{Canonical basis of $U_q^-(\frakg)$}
\label{section:canonical-basis-definition}

Kashiwara \cite{Kashiwara:CrystalBases-QuantizedUniversalEnvelopingAlgebras} showed that for any $i\in I$
\[
    U_q^-(\frakg) = \bigoplus_{n\ge 0} f_i^{(n)} \ker e_i'.
\]
Let $\te_i, \tf_i\colon U_q^-(\frakg) \to U_q^-(\frakg)$ be the operators defined by
\[
    \te_i(f_i^{(n)}u) = f_i^{(n-1)}u
    \text{ and }
    \tf_i(f_i^{(n)}u) = f_i^{(n+1)}u
    \text{ for } u \in \ker e_i'.
\]
They are commonly referred to as \emph{Kashiwara operators}.

Let $L(\infty)$ be the $A$-submodule of $U_q^-(\frakg)$ generated by $\tf_{i_1}\cdots \tf_{i_r} \cdot 1$ for $r \ge 0$ and $i_1, \dots, i_r \in I$.
Denote by $B(\infty) \subseteq L(\infty)/qL(\infty)$ the images of vectors of the form $\tf_{i_1}\cdots \tf_{i_r} \cdot 1$ and write $b_\infty \in B(\infty)$ for the image of $1$ under the projection $L(\infty) \to L(\infty)/qL(\infty)$.

There is an isomorphism \cite[Theorem 6]{Kashiwara:CrystalBases-QuantizedUniversalEnvelopingAlgebras}
\[
    G\colon L(\infty)/qL(\infty) \to U_\bbQ^-(\frakg) \cap L(\infty) \cap \overline{L(\infty)}
\]
of $\bbQ$-vector spaces.

The image of $B(\infty)$ under $G$ is called the \emph{(lower) global crystal basis} or \emph{canonical basis} of $U_q^-(\frakg)$; we denote it by $\calB(\infty)$.
We refer to its dual basis $\calB(\infty)^\vee$ in the graded dual $U_q^-(\frakg)^\vee$ as the \emph{dual canonical basis}.

\begin{theorem}[{\cite[Theorem 2.1.1]{Kashiwara:CrystalBase-Littelmann-DemazureCharacterFormula}}]
    \label{theorem:crystal-base-star-invariance}
    The crystal base $B(\infty)$ is invariant under the $*$-involution, $*B(\infty) = B(\infty)$.
\end{theorem}

\begin{corollary}
    \label{corollary:canonical-basis-star-invariance}
    The canonical basis $\calB(\infty)$ is invariant under the $*$-involution; more precisely, $*G(b) = G(*b)$ for any $b \in B(\infty)$.
\end{corollary}

From Propositions 5.1.2 and 5.1.3 in \cite{Kashiwara:CrystalBases-QuantizedUniversalEnvelopingAlgebras} and the definition of the canonical basis, we obtain the following criterion.
\begin{theorem}
    \label{theorem:canonical-basis-criterion}
    Let $u \in U_q^-(\frakg)$. If $u \in U_\bbZ^-(\frakg)$, $\bar u = u$ and $(u, u) \in 1+qA$, then either $u$ or $-u$ is in $\calB(\infty)$.
\end{theorem}

\section{Tight monomial cones}
\label{section:tight-monomial-cones}

In this section we recall the definition of Lusztig's tight monomial cone $L_\uli$ and the adapted string cone $S_\uli(\infty)$ defined by Littelmann \cite{Littelmann:Cones-Crystals-Patterns}, and show that $L_\uli$ embeds into $S_\uli(\infty)$.
This allows us to give a sufficient condition for a monomial in $L_\uli$ to belong to the canonical basis (Proposition \ref{proposition:lusztig-cone-canonical-basis}).
We then use this condition to find a family of monomials which belong to the canonical basis (Proposition \ref{proposition:family-of-monomials-in-canonical-basis}).

Let $\uli = (i_1, \dots, i_r)$ be a sequence in $I$.
The \emph{tight monomial cone} $L_\uli \subseteq \bbZ_{\ge0}^r$ is the cone consisting of tuples $(a_1, \dots, a_r)\in \bbZ_{\ge0}^r$ such that for every pair $1 \le p < q \le r$ with $i_p = i_q := i$ we have
\begin{equation}
    \label{equation:lusztig-cone-condition}
    a_p + a_q + \sum_{p<k<q} \langle h_i, \alpha_{i_k}\rangle a_k \le 0.
\end{equation}
We refer to this as the \emph{tight monomial cone condition}.
For convenience, we also say that a monomial $f_{\uli}^{(\ula)} = f_{i_1}^{(a_1)} \cdots f_{i_r}^{(a_r)}$ belongs to the tight monomial cone $L_\uli$ if $(a_1, \dots, a_r) \in L_\uli$.

\subsection{Embedding into adapted string cones}
Let $w\in W$ and let $\uli = (i_1, \dots, i_r)$ be a sequence in $I$ such that $s_\uli$ a reduced decomposition of $w$.
A vector $\ula = (a_1, \dots, a_r) \in \bbZ_{\ge0}^r$ is called an \emph{adapted string} of $b\in B(\lambda)$ if
\[
    b = \tf_{i_1}^{a_1} \cdots \tf_{i_r}^{a_r} b_\lambda
\]
and for all $k \in [1, r]$
\[
    \te_{i_k} \tf_{i_{k+1}}^{a_{k+1}} \cdots \tf_{i_r}^{a_r} b_\lambda = 0.
\]
We denote by $S_\uli(\lambda)$ the set of all $\ula = (a_1, \dots, a_r) \in \bbZ_{\ge0}^r$ such that $\ula$ is an adapted string of some $b\in B(\lambda)$; and we put $S_\uli(\infty) = \bigcup S_\uli(\lambda)$, where the union ranges over all dominant weights $\lambda$.

Let $C_\uli(\infty)$ be the real cone spanned by $S_\uli(\infty)$ and let $C_\uli(\lambda)$ be the real polytope defined by the inequalities
\begin{equation}
    \label{equation:string-polytope-condition}
    a_p \le \langle h_{i_p}, \lambda\rangle - \sum_{p<k\le r} \langle h_{i_p}, \alpha_{i_k}\rangle a_k \quad \text{for } p \in [1, r].
\end{equation}

\begin{proposition}[{\cite[Proposition 1.5]{Littelmann:Cones-Crystals-Patterns}}]
    The following holds:
    \begin{enumerate}
        \item $C_\uli(\infty)$ is a rational cone and $S_\uli(\infty)$ is the set of its integral points.
        \item $C_\uli(\lambda)$ is a rational polytope and $S_\uli(\lambda)$ is the set of its integral points.
    \end{enumerate}
\end{proposition}

Let $>$ denote the lexicographic order on $S_\uli(\infty)$ and $S_\uli(\lambda)$.

\begin{proposition}[{\cite[Proposition 10.3]{Littelmann:Cones-Crystals-Patterns}}]
    \label{proposition:base-change-canonical-basis-string-basis}
    Let $w\in W$ and let $\uli = (i_1, \dots, i_r)$ be a sequence in $I$ such that $s_\uli$ a reduced decomposition of $w$.
    Let $b \in B_w(\infty)$ and let $\ula = (a_1, \dots, a_r)$ be its adapted string.
    Then,
    \[
        f_{i_1}^{(a_1)} \cdots f_{i_r}^{(a_r)} = G(b) + \sum z_{b,b'} G(b'),
    \]
    where $z_{b,b'} \in qA$ is zero unless $\ula' > \ula$ for the adapted string $\ula'$ of $b'$.

    In particular, the set
    \[
        \{ f_{i_1}^{(a_1)} \cdots f_{i_r}^{(a_r)} \mid (a_1, \dots, a_r) \in S_\uli(\infty) \}
    \]
    is a basis of $U_q^-(\frakg, w)$.
\end{proposition}

\begin{proposition}
    \label{proposition:lusztig-cone-string-cone}
    Let $\uli = (i_1, \dots, i_r)$ be a reduced word.
    Then, $L_\uli \subseteq S_\uli(\infty)$.
\end{proposition}
\begin{proof}
    Let $\ula = (a_1, \dots, a_r) \in L_\uli$.
    We construct a dominant weight $\lambda$ such that $\ula$ is an element of the string polytope $S_\uli(\lambda)$; then the proposition follows.
    For $i\in I$, define $\lambda_i := \max_{i_k=i} a_k$ (with the convention that $\lambda_i = 0$ if $i$ does not occur in $\uli$) and put $\lambda := \sum_{i\in I} \lambda_i \varpi_i$.
    Let $p \in [1, r]$ and let $q$ be maximal such that $i_q = i_p$.
    If $p = q$, then
    \[
        a_p \le \langle h_{i_p}, \lambda\rangle \le \langle h_{i_p}, \lambda\rangle - \sum_{p<k\le r} \langle h_{i_p}, \alpha_{i_k}\rangle a_k;
    \]
    and if $p < q$, we have
    \begin{align*}
        a_p &\le -a_q - \sum_{p<k<q} \langle h_{i_p}, \alpha_{i_k}\rangle a_k \\
        &\le \langle h_{i_p}, \lambda\rangle -2a_q - \sum_{p<k<q} \langle h_{i_p}, \alpha_{i_k}\rangle a_k \\
        &\le \langle h_{i_p}, \lambda\rangle - \sum_{p<k\le r} \langle h_{i_p}, \alpha_{i_k}\rangle a_k.
    \end{align*}
    This shows that the all inequalities (\ref{equation:string-polytope-condition}) for the string polytope are satisfied.
\end{proof}

\begin{proposition}
    \label{proposition:lusztig-cone-canonical-basis}
    Let $\uli = (i_1, \dots, i_r)$ be a reduced word.
    For $(a_1, \dots, a_r) \in L_\uli$ assume that $(f_\uli^{(\ula)}, f_\uli^{(\ula)}) \in 1 + qA$.
    Then, $f_\uli^{(\ula)}$ is an element of the canonical basis of $U_q^-(\frakg)$.
\end{proposition}
\begin{proof}
    By definition we have $f_\uli^{(\ula)} \in U_\bbZ^-(\frakg)$ and $\overline{f_\uli^{(\ula)}} = f_\uli^{(\ula)}$.
    Together with the assumption that $(f_\uli^{(\ula)}, f_\uli^{(\ula)}) \in 1 + qA$, this means all conditions of Theorem \ref{theorem:canonical-basis-criterion} are fulfilled.
    Thus, $f_\uli^{(\ula)}$ is an element of the canonical basis up to a sign.
    Since $\uli$ is a reduced word, Proposition \ref{proposition:lusztig-cone-string-cone} implies that $(a_1, \dots, a_r) \in S_\uli(\infty)$.
    Finally, Proposition \ref{proposition:base-change-canonical-basis-string-basis} allows us to conclude the proposition.
\end{proof}

We end this section by presenting two technical lemmas, which give us a criterion on $L_\uli$ for $\uli$ to be a reduced word.

\begin{lemma}
    \label{lemma:reduced-decomposition-dominant-weight}
    Let $\uli = (i_1, \dots, i_r)$ be a sequence in $I$.
    There exists a weight $\lambda \in P$ such that
    \[
        \langle h_{i_k}, s_{i_{k+1}}\cdots s_{i_r}\lambda\rangle > 0
    \]
    for all $1\le k \le r$ if and only if $\uli$ is a reduced word.
\end{lemma}
\begin{proof}
    For an element $w\in W$ it is a well-known fact that $\ell(s_iw) > \ell(w)$ if and only if $w^{-1}\alpha_i$ is a positive root.
    Thus, if $s_\uli$ is a reduced decomposition, we can choose any dominant regular weight $\lambda$.
    Indeed, $\ell(s_{i_k}\cdots s_{i_r}) > \ell(s_{i_{k+1}}\cdots s_{i_r})$ is equivalent to $\langle s_{i_r}\cdots s_{i_{k+1}} h_{i_k}, \lambda\rangle > 0$ (since $\lambda$ is dominant and regular), that is $\langle h_{i_k}, s_{i_{k+1}}\cdots s_{i_r}\lambda\rangle > 0$.

    Conversely, assume that there exists a weight $\lambda \in P$ such that $\langle h_{i_k}, s_{i_{k+1}}\cdots s_{i_r}\lambda\rangle > 0$ for all $1\le k \le r$.
    Then, $\lambda$ is conjugate to a dominant weight $\lambda^+$ by some $w\in W$.
    Fix a reduced decomposition of $w = s_{j_1} \cdots s_{j_q}$.
    Using the same argument as above we see that $s_\uli (s_{j_1} \cdots s_{j_q})$ is a reduced decomposition of $s_\uli w$.
    In particular, $\uli$ is a reduced word.
\end{proof}

\begin{lemma}
    \label{lemma:lusztig-cone-reduced-decomposition}
    Assume that $\frakg$ is of finite type.
	Let $\uli = (i_1, \dots , i_r)$ be a sequence in $I$.
	Assume that there exists $(a_1, \dots, a_r) \in L_\uli$ such that $a_k > 0$ for all $k \in [1, r]$ and such that for every pair $p < q$ with $i_p = i_q =: i$ we have
	\[
	    a_p + a_q + \sum_{p < k < q} \langle h_i, \alpha_{i_k}\rangle a_k = 0.
	\]
	Then, $\uli$ is a reduced word.
\end{lemma}
\begin{proof}
    For $i\in I$, put $\lambda_i := 0$ if $i_k \neq i$ for all $k \in [1, r]$, otherwise put $\lambda_i := a_p + \sum_{p<k\le r} \langle h_{i_p}, \alpha_{i_k}\rangle a_k$, where $p$ is chosen maximally with $i_p = i$.
    Then, we have
    \[
        \langle h_{i_k}, s_{i_{k+1}}\cdots s_{i_r}\lambda\rangle = a_k > 0 \quad \text{for all } k \in [1, r].
    \]
    Now, the claim follows from Lemma \ref{lemma:reduced-decomposition-dominant-weight}.
\end{proof}

\subsection{A family of monomials in the canonical basis}
\label{section:monomials-in-canonical-basis}

Let $\calM$ the subset of $U_q^-(\frakg)$ consisting of the monomials $f_{i_1}^{(a_1)}\cdots f_{i_n}^{(a_n)}$ satisfying the following properties:
\begin{enumerate}
    \item $\uli := (i_1, \dots, i_n)$ is a reduced word,
    \item $(a_1, \dots, a_n) \in L_\uli$,
    \item every $i \in I$ occurs in $\uli$ at most twice, and
    \item if $i_1 = i_r$ for some $r > 1$, then
    \[
        f_{i_2}^{(a_2)}\cdots f_{i_r}^{(a_r-a_1)}\cdots f_{i_n}^{(a_n)} \in \calM.
    \]
\end{enumerate}
Let $\nu \in Q_-$, we define $\calM_\nu := \calM \cap U_q^-(\frakg)_\nu$.

\begin{proposition}
    \label{proposition:family-of-monomials-in-canonical-basis}
    Every element of $\calM$ belongs to the canonical basis.
\end{proposition}
\begin{proof}
    Let $\nu \in Q_-$ and $u \in \calM_\nu$. We prove the proposition by induction on $\trace(\nu)$.
    If $\trace(\nu) = 0$, clearly $u = 1$ and the statement is true.
    Now, let $\trace(\nu) < 0$ and assume that the claim is true for all $u' \in \calM_{\nu'}$ with $\trace(\nu') > \trace(\nu)$.
    Since $\uli$ is a reduced word, by Proposition \ref{proposition:lusztig-cone-canonical-basis} it is sufficient to check that $(u, u) \in 1+qA$.
    For this write $u = f_{i_1}^{(a_1)}\cdots f_{i_n}^{(a_n)}$ as in the definition of $\calM$.

    By property (iii) $i_1$ occurs at most twice in $\uli$. First, assume that $i_1$ only occurs once in $\uli$. Then,
    \begin{align*}
        & (f_{i_1}^{(a_1)}\cdots f_{i_n}^{(a_n)}, f_{i_1}^{(a_1)}\cdots f_{i_n}^{(a_n)}) \\
        ={}& [a_1]_{i_1}!^{-1} (f_{i_2}^{(a_2)}\cdots f_{i_n}^{(a_n)}, e_{i_1}'^{a_1} f_{i_1}^{(a_1)}\cdots f_{i_n}^{(a_n)}) \\
        ={}& [a_1]_{i_1}!^{-1} q_{i_1}^{-2a_1^2 + 2a_1^2 - a_1(a_1-1)/2} (f_{i_2}^{(a_2)}\cdots f_{i_n}^{(a_n)}, f_{i_2}^{(a_2)}\cdots f_{i_n}^{(a_n)})
    \end{align*}
    Observe that $u' := f_{i_2}^{(a_2)}\cdots f_{i_n}^{(a_n)} \in \calM$ and $\trace(\weight u') > \trace(\weight u)$.
    Using the induction hypothesis, we have $(u', u') \in 1 + qA$ and since $[a_1]_{i_1}!^{-1} q_{i_1}^{-a_1(a_1-1)/2} \in 1 + qA$, the induction proceeds in this case.

    Now, consider the case where $i_1$ occurs twice in $\uli$ and let $r$ denote the second index with $i_r = i_1$. Put $\gamma = \sum_{t=2}^{r-1} \langle h_{i_1}, \alpha_{i_t}\rangle a_t$. We compute
	\begin{align*}
		& (f_{i_1}^{(a_1)}\cdots f_{i_n}^{(a_n)}, f_{i_1}^{(a_1)}\cdots f_{i_n}^{(a_n)}) \\
		={}& [a_1]_i!^{-1} (f_{i_2}^{(a_2)}\cdots f_{i_n}^{(a_n)}, e_{i_1}'^{a_1}f_{i_1}^{(a_1)}\cdots f_{i_n}^{(a_n)}) \\
		={}& \sum_{k=0}^{a_1} [a_1]_{i_1}!^{-1} q_{i_1}^{-2a_1^2+2ka_1-k(k-1)/2} \qbinom{a_1}{k}_{i_1} \\
        &\times (f_{i_2}^{(a_2)}\cdots f_{i_n}^{(a_n)}, f_{i_1}^{(a_1-k)} e_{i_1}'^{a_1-k} f_{i_2}^{(a_2)} \cdots f_{i_n}^{(a_n)}) \\
		={}& \sum_{k=0}^{a_1} [a_1]_{i_1}!^{-1} [a_1-k]_{i_1}!^{-1} q_{i_1}^{-2a_1^2+2ka_1-k(k-1)/2} \qbinom{a_1}{k}_{i_1} \\
        &\times (e_{i_1}'^{a_1-k} f_{i_2}^{(a_2)}\cdots f_{i_n}^{(a_n)}, e_{i_1}'^{a_1-k} f_{i_2}^{(a_2)} \cdots f_{i_n}^{(a_n)}) \\
		={}& \sum_{k=0}^{a_1} [a_1]_{i_1}!^{-1} [a_1-k]_{i_1}!^{-1} q_{i_1}^{-2a_1^2+2ka_1-k(k-1)/2} \qbinom{a_1}{k}_{i_1} \\
		&\times q_{i_1}^{-2(a_1-k)\gamma} q_{i_1}^{-4(a_1-k)a_r + 2(a_1-k)(a_1-k+a_r)-(a_1-k)(a_1-k-1)} \\
		&\times (f_{i_2}^{(a_2)}\cdots f_{i_r}^{(a_r-a_1+k)} \cdots f_{i_n}^{(a_n)}, f_{i_2}^{(a_2)} \cdots f_{i_r}^{(a_r-a_1+k)} \cdots f_{i_n}^{(a_n)}).
	\end{align*}
	We have
    \[
        q_{i_1}^{-k(k-1)/2 - (a_1-k)(a_1-k-1)} [a_1]_{i_1}!^{-1} [a_1-k]_{i_1}!^{-1} \qbinom{a_1}{k}_{i_1} \in 1 + qA
    \]
    and by property (iv) we know that $u_k' := f_{i_2}^{(a_2)}\cdots f_{i_r}^{(a_r-a_1+k)}\cdots f_{i_n}^{(a_n)} \in \calM$ for any $0 \le k \le a_1$ and $\trace(\weight u_k') > \trace(\weight u)$.
    Thus, the induction hypothesis can be applied to $u_k'$ and one obtains $(u_k', u_k') \in 1+qA$.
	The remaining powers of $q_{i_1}$ can be written as
	\[
		\eta = q_{i_1}^{2(a_1-k)^2 - 2(a_1-k)(a_1+a_r-\gamma)}.
	\]
    From property (ii), we obtain $a_1+a_r-\gamma \le 0$; it follows that $\eta = 1$ for $k = a_1$ and $\eta \in qA$ for $k < a_1$. Thus, only the summand for $k = a_1$ is in $1 + qA$ and the remaining ones are in $qA$. Hence, $(u, u) \in 1 + qA$ and the induction proceeds.
\end{proof}

\begin{example}
    \label{example:monomials-in-canonical-basis}
    Let $\{ i_1, \dots, i_p\}$ and $\{ j_1, \dots, j_q\}$ be disjoint subsets of $I$.
    Let $\uli$ be a shuffle of $(i_1, \dots, i_p, i_p, \dots, i_1)$ and $(j_1, \dots, j_q)$, and let $\ula = (a_1, \dots, a_{2p+q}) \in L_\uli$.
    Then $f_\uli^{(\ula)}$ is an element of $\calM$ and therefore in the canonical basis.

    Note that $\uli$ is not necessarily reduced, however the subsequence consisting of the $i_k$ with $a_k \neq 0$ is.
    This follows from Lemma \ref{lemma:lusztig-cone-reduced-decomposition} after modifying $\ula$ appropriately.
\end{example}

\section{Dual imaginary vectors}
\label{section:dual-imaginary-vectors}

An element of the dual canonical basis $b\in \calB(\infty)^\vee$ is called \emph{real} if $b^2\in q^\bbZ\calB(\infty)^\vee$, otherwise \emph{imaginary}.
Leclerc gave an example of an imaginary vector for each of the types $A_5$, $B_3$, $C_3$, $D_4$, and $G_2$ in Section 2 of \cite{Leclerc:ImaginaryVectors-DualCanonicalBasis}.
Assume that $\frakg$ is one of these types and that $b\in \calB(\infty)^\vee$ is the corresponding example.
Let $\theta \in \calB(\infty)$ be the unique element with $(b, \theta) = 1$.
We claim that it is a monomial and has the following shape:
\begin{equation}
    \label{equation:notation-theta}
    \theta = (f_{i_1}^{(a_1)}\cdots f_{i_p}^{(a_p)}) (f_{j_1}^{(b_1)}\cdots f_{j_q}^{(b_q)}) (f_{i_1}^{(a_1)}\cdots f_{i_p}^{(a_p)}).
\end{equation}
Indeed, the element $\theta$ can be obtained from $b$ in the following way:
Using the fact that $B(\infty)$ is an orthonormal basis with respect to the bilinear form on $L(\infty)/qL(\infty)$ (which is induced by that of $U_q^-(\frakg)$), one sees that it is sufficient to know the Lusztig datum (PBW datum) of $b$ to find $\theta$.
This information was already given by Leclerc in the same paper \cite{Leclerc:ImaginaryVectors-DualCanonicalBasis}.
From here, one uses the crystal structure on Lusztig data to compute the adapted string of $b$ with respect to the same reduced decomposition as the one used by Leclerc.
Finally, one can use Proposition \ref{proposition:family-of-monomials-in-canonical-basis} to see that the monomial corresponding to the adpated string is indeed in the canonical basis.

\begin{remark}
    For the crystal structure on Lusztig data we refer the reader to \cite{BerensteinZelevinsky:TensorProductMultiplicities-CanonicalBases-TotallyPositiveVarieties}.
\end{remark}

Similarly to the concept of real and imaginary vectors, we consider the ``square'' of $\theta$:
\begin{equation}
    \label{equation:notation-xi}
    \begin{aligned}
        \xi ={}& (f_{i_1}^{(a_1)}\cdots f_{i_p}^{(a_p)}) (f_{j_1}^{(b_1)}\cdots f_{j_q}^{(b_q)}) (f_{i_1}^{(2a_1)}\cdots f_{i_p}^{(2a_p)}) \\
        &\cdot (f_{j_1}^{(b_1)}\cdots f_{j_q}^{(b_q)}) (f_{i_1}^{(a_1)}\cdots f_{i_p}^{(a_p)}).
    \end{aligned}
\end{equation}
A close observation of $\theta$ and $\xi$ reveals the following properties:
\begin{enumerate}[label=(P\arabic*)]
    \item The sets $\{i_1 \dots, i_p\}$ and $\{j_1, \dots, j_q\}$ form a partition of $I$,
    \item $(\alpha_{i_r}, \alpha_{i_s}) = 0$ for all $1\le r < s \le p$ and $(\alpha_{j_r}, \alpha_{j_s}) = 0$ for all $1\le r < s \le q$,
    \item $3a_k + \sum_{t=1}^q \langle h_{i_k}, \alpha_{j_t}\rangle b_{j_t} = 0$ for $1 \le k \le p$, and
    \item $b_k + \sum_{t=1}^p \langle h_{j_k}, \alpha_{i_t}\rangle a_{i_t} = 0$ for $1\le k \le q$.
\end{enumerate}

\begin{remark}
    \begin{enumerate}
        \item Property (P2) implies that $\xi$ is indeed a square of $\theta$ up to $q$-binomials.
        \item Properties (P1) and (P2) imply that
            \[
                |\{i_1 \dots, i_p\}| = p \quad \text{and} \quad |\{j_1 \dots, j_q\}| = q.
            \]
            In particular, each $i\in I$ occurs exactly once in either $i_1, \dots, i_p$ or $j_1, \dots, j_q$.
        \item Properties (P3) and (P4) are the tight condition in (\ref{equation:lusztig-cone-condition}) for $\xi$.
    \end{enumerate}
\end{remark}

\subsection{Classification and properties}
This leads to the following classification problem:
Determine all finite-dimen\-sional simple complex Lie algebras $\frakg$, together with pairs of monomials $(\theta, \xi)$, as above, such that properties (P1) - (P4) hold.

\begin{theorem}
    \label{theorem:classification}
    Every triple $(\frakg, \theta, \xi)$ fulfilling the properties (P1) - (P4) belongs to one of the following 10 families:
    \begin{enumerate}
    \item Type $A_5$ \\[\abovedisplayskip]
    \(
        \begin{aligned}
            \theta_1^{[l]} &= (f_2^{(l)}f_4^{(l)}) (f_1^{(l)}f_3^{(2l)}f_5^{(l)}) (f_2^{(l)}f_4^{(l)}) \\
            \xi_1^{[l]} &= (f_2^{(l)}f_4^{(l)}) (f_1^{(l)}f_3^{(2l)}f_5^{(l)}) (f_2^{(2l)}f_4^{(2l)}) (f_1^{(l)}f_3^{(2l)}f_5^{(l)}) (f_2^{(l)}f_4^{(l)})
        \end{aligned}
    \)
    \vspace\belowdisplayskip
    \item Type $A_5$ \\[\abovedisplayskip]
    \(
        \begin{aligned}
            \theta_2^{[l]} &= (f_1^{(l)}f_3^{(2l)}f_5^{(l)}) (f_2^{(3l)}f_4^{(3l)}) (f_1^{(l)}f_3^{(2l)}f_5^{(l)}) \\
            \xi_2^{[l]} &= (f_1^{(l)}f_3^{(2l)}f_5^{(l)}) (f_2^{(3l)}f_4^{(3l)}) (f_1^{(2l)}f_3^{(4l)}f_5^{(2l)}) (f_2^{(3l)}f_4^{(3l)}) (f_1^{(l)}f_3^{(2l)}f_5^{(l)})
        \end{aligned}
    \)
    \vspace\belowdisplayskip
    \item Type $B_3$ \\[\abovedisplayskip]
    \(
        \begin{aligned}
            \theta_3^{[l]} &= f_2^{(l)} (f_1^{(l)}f_3^{(2l)}) f_2^{(l)} \\
            \xi_3^{[l]} &= f_2^{(l)} (f_1^{(l)}f_3^{(2l)}) f_2^{(2l)} (f_1^{(l)} f_3^{(2l)}) f_2^{(l)}
        \end{aligned}
    \)
    \vspace\belowdisplayskip
    \item Type $B_3$ \\[\abovedisplayskip]
    \(
        \begin{aligned}
            \theta_4^{[l]} &= (f_1^{(l)}f_3^{(2l)}) f_2^{(3l)} (f_1^{(l)}f_3^{(2l)}) \\
            \xi_4^{[l]} &= (f_1^{(l)}f_3^{(2l)}) f_2^{(3l)} (f_1^{(2l)}f_3^{(4l)}) f_2^{(3l)} (f_1^{(l)}f_3^{(2l)})
        \end{aligned}
    \)
    \vspace\belowdisplayskip
    \item Type $C_3$ \\[\abovedisplayskip]
    \(
        \begin{aligned}
            \theta_5^{[l]} &= f_2^{(l)} (f_1^{(l)}f_3^{(l)}) f_2^{(l)} \\
            \xi_5^{[l]} &= f_2^{(l)} (f_1^{(l)}f_3^{(l)}) f_2^{(2l)} (f_1^{(l)}f_3^{(l)}) f_2^{(l)}
        \end{aligned}
    \)
    \vspace\belowdisplayskip
    \item Type $C_3$ \\[\abovedisplayskip]
    \(
        \begin{aligned}
            \theta_6^{[l]} &= (f_1^{(l)}f_3^{(l)}) f_2^{(3l)} (f_1^{(l)}f_3^{(l)}) \\
            \xi_6^{[l]} &= (f_1^{(l)}f_3^{(l)}) f_2^{(3l)} (f_1^{(2l)}f_3^{(2l)}) f_2^{(3l)} (f_1^{(l)}f_3^{(l)})
        \end{aligned}
    \)
    \vspace\belowdisplayskip
    \item Type $D_4$ \\[\abovedisplayskip]
    \(
        \begin{aligned}
            \theta_7^{[l]} &= f_2^{(l)} (f_1^{(l)}f_3^{(l)}f_4^{(l)}) f_2^{(l)} \\
            \xi_7^{[l]} &= f_2^{(l)} (f_1^{(l)}f_3^{(l)}f_4^{(l)}) f_2^{(2l)} (f_1^{(l)}f_3^{(l)}f_4^{(l)}) f_2^{(l)}
        \end{aligned}
    \)
    \vspace\belowdisplayskip
    \item Type $D_4$ \\[\abovedisplayskip]
    \(
        \begin{aligned}
            \theta_8^{[l]} &= (f_1^{(l)}f_3^{(l)}f_4^{(l)}) f_2^{(3l)} (f_1^{(l)}f_3^{(l)}f_4^{(l)}) & \\
            \xi_8^{[l]} &= (f_1^{(l)}f_3^{(l)}f_4^{(l)}) f_2^{(3l)} (f_1^{(2l)}f_3^{(2l)}f_4^{(2l)}) f_2^{(3l)} (f_1^{(l)}f_3^{(l)}f_4^{(l)}) &
        \end{aligned}
    \)
    \vspace\belowdisplayskip
    \item Type $G_2$ \\[\abovedisplayskip]
    \(
        \begin{aligned}
            \theta_9^{[l]} &= f_1^{(l)} f_2^{(l)} f_1^{(l)} \\
            \xi_9^{[l]} &= f_1^{(l)} f_2^{(l)} f_1^{(2l)} f_2^{(l)} f_1^{(l)}
        \end{aligned}
    \)
    \vspace\belowdisplayskip
    \item Type $G_2$ \\[\abovedisplayskip]
    \(
        \begin{aligned}
            \theta_{10}^{[l]} &= f_2^{(l)} f_1^{(3l)} f_2^{(l)} \\
            \xi_{10}^{[l]} &= f_2^{(l)} f_1^{(3l)} f_2^{(2l)} f_1^{(3l)} f_2^{(l)}
        \end{aligned}
    \)
    \end{enumerate}
\end{theorem}
\begin{proof}
    First, observe that (P1) and (P2) together imply that there exactly two ways to label the Dynkin diagram of $\frakg$ with $\{i_1, \dots, i_p\}$ and $\{j_1, \dots, j_q\}$ and second, observe that (P3) and (P4) define a linear system over $\bbR^{p+q}$ with variables $a_1, \dots, a_p, b_1, \dots, b_q$.
    Determining a positive integral solution to linear system is thus equivalent to finding a monomial fulfilling (P1) - (P4).
    Once a labeling is chosen, it is straightforward to solve the linear system. Since we only care about positive integral solutions, we can abort our computations as soon as one coordinate is $\le 0$ and a solution will always mean positive integral solution. We now proceed with a case-by-case analysis.

    \emph{Type $A_n$}. First, consider the labeling
    \[
        \begin{tikzpicture}[baseline=(j1.base)]
            \node (j1) at (0,0) {$j_1$};
            \node (i1) at (1,0) {$i_1$};
            \node (j2) at (2,0) {$j_2$};
            \node (i2) at (3,0) {$i_2$};
            \node (j3) at (4,0) {$\dots$};
            \draw (j1) -- (i1) -- (j2) -- (i2) -- (j3);
        \end{tikzpicture}
    \]
    (P3) and (P4) give the following system of linear equations:
    \[
        \left.
        \begin{aligned}
            b_1 - a_1 &= 0 \\
            3a_1 - b_1 - b_2 &= 0 \\
            b_2 - a_1 - a_2 &= 0 \\
            3a_2 - b_2 - b_3 &= 0 \\
            b_3 - a_2 - a_3 &= 0 \\
        \end{aligned}
        \quad
        \right\rbrace
        \quad
        \begin{aligned}
            b_1 &= a_1 \\
            b_2 &= 3a_1 - b_1 = 2a_1 \\
            a_2 &= b_2 - a_1 = a_1 \\
            b_3 &= 3a_2 - b_2 = a_1 \\
            a_3 &= b_3 - a_2 = 0
        \end{aligned}
    \]
    Observe that taking $p < 2$ or $q < 3$ only admits the trivial solution for the linear system; we also cannot take $p \ge 3$, since in this case $a_3 = 0$ would not be positive, and from (P1) it follows that $q$ cannot be greater than $3$.
    Thus, the linear system only has a solution if $p = 2$, $q = 3$, $\{i_1, i_2\} = \{ 2, 4\}$ and $\{j_1,j_2,j_3\} = \{1,3,5\}$. Fixing the order $(i_1, i_2, j_1, j_2, j_3) = (2,4,1,3,5)$ gives the solutions $(l,l,l,2l,l)$ for $l\in \bbN$.

    The second possible labeling is
    \[
        \begin{tikzpicture}[baseline=(j1.base)]
            \node (n1) at (0,0) {$i_1$};
            \node (n2) at (1,0) {$j_1$};
            \node (n3) at (2,0) {$i_2$};
            \node (n4) at (3,0) {$j_2$};
            \node (n5) at (4,0) {$\dots$};
            \draw (n1) -- (n2) -- (n3) -- (n4) -- (n5);
        \end{tikzpicture}
    \]
    (P3) and (P4) give the following system of linear equations:
    \[
        \left.
        \begin{aligned}
            3a_1 - b_1 &= 0 \\
            b_1 - a_1 - a_2 &= 0 \\
            3a_2 - b_1 - b_2 &= 0 \\
            b_2 - a_2 - a_3 &= 0 \\
            3a_3 - b_2 - b_3 &= 0 \\
        \end{aligned}
        \quad
        \right\rbrace
        \quad
        \begin{aligned}
            b_1 &= 3a_1 \\
            a_2 &= b_1 - a_2 = 2a_1 \\
            b_2 &= 3a_2 - b_1 = 3a_1 \\
            a_3 &= b_2 - a_2 = a_1 \\
            b_3 &= 3a_3 - b_2 = 0
        \end{aligned}
    \]
    It only has a solution if $p = 3$, $q = 2$, $\{ i_1, i_2, i_3\} = \{ 1, 3, 5 \}$ and $\{ j_1, j_2 \} = \{ 2, 4\}$. Fixing $(i_1, i_2, i_3, j_1, j_2) = (1,3,5,2,4)$ gives the solutions $(l,2l,l,3l,3l)$ for $l\in \bbN$.

    \emph{Type $B_n$}. First, consider the labeling
    \[
        \begin{tikzpicture}
            \node (n1) at (0,0) {$\dots$};
            \node (n2) at (1,0) {$i_2$};
            \node (n3) at (2,0) {$j_2$};
            \node (n4) at (3,0) {$i_1$};
            \node (n5) at (4,0) {$j_1$};
            \draw (n1) -- (n2) -- (n3) -- (n4);
            \drawDoubleEdge{n4}{n5}
        \end{tikzpicture}
    \]
    (P3) and (P4) give the following system of linear equations:
    \[
        \left.
        \begin{aligned}
            b_1 - 2a_1 = 0 \\
            3a_1 - b_1 - b_2 = 0 \\
            b_2 - a_1 - a_2 = 0
        \end{aligned}
        \quad
        \right\rbrace
        \quad
        \begin{aligned}
            b_1 &= 2a_1 \\
            b_2 &= 3a_1 - b_1 = a_1 \\
            a_2 &= b_2 - a_1 = 0
        \end{aligned}
    \]
    It only has a solution if $p = 1$, $q = 2$, $i_1 = 2$, $\{ j_1, j_2 \} = \{ 1, 3 \}$. Fixing $(i_1, j_1, j_2) = (2, 1, 3)$ gives the solutions $(l,l,2l)$ for $l\in \bbN$.

    The second possible labeling is
    \[
        \begin{tikzpicture}
            \node (n1) at (0,0) {$\dots$};
            \node (n2) at (1,0) {$j_2$};
            \node (n3) at (2,0) {$i_2$};
            \node (n4) at (3,0) {$j_1$};
            \node (n5) at (4,0) {$i_1$};
            \draw (n1) -- (n2) -- (n3) -- (n4);
            \drawDoubleEdge{n4}{n5}
        \end{tikzpicture}
    \]
    (P3) and (P4) give the following system of linear equations:
    \[
        \left.
        \begin{aligned}
            3a_1 - 2b_1 &= 0 \\
            b_1 - a_1 - a_2 &= 0 \\
            3a_2 - b_1 - b_2 &= 0
        \end{aligned}
        \quad
        \right\rbrace
        \quad
        \begin{aligned}
            b_1 &= \tfrac32 a_1 \\
            a_2 &= b_1 - a_1 = \tfrac12 a_1 \\
            b_2 &= 3a_2 - b_1 = 0
        \end{aligned}
    \]
    It only has a solution if $p = 2$, $q = 1$ and $\{ i_1, i_2 \} = \{ 1, 3 \}$, and $j_1 = 2$. Fixing $(i_1,i_2,j_1) = (1,3,2)$ gives the solutions $(l,2l,3l)$ for $l\in \bbN$.

    \emph{Type $C_n$}. First, consider the labeling
    \[
        \begin{tikzpicture}
            \node (n1) at (0,0) {$\dots$};
            \node (n2) at (1,0) {$i_2$};
            \node (n3) at (2,0) {$j_2$};
            \node (n4) at (3,0) {$i_1$};
            \node (n5) at (4,0) {$j_1$};
            \draw (n1) -- (n2) -- (n3) -- (n4);
            \drawReverseDoubleEdge{n4}{n5}
        \end{tikzpicture}
    \]
    (P3) and (P4) give the following system of linear equations:
    \[
        \left.
        \begin{aligned}
            b_1 - a_1 &= 0 \\
            3a_1 - 2b_1 - b_2 &= 0 \\
            b_2 - a_1 - a_2 &= 0
        \end{aligned}
        \quad
        \right\rbrace
        \quad
        \begin{aligned}
            b_1 &= a_1 \\
            b_2 &= 3a_1 - 2b_1 = a_1 \\
            a_2 &= b_2 - a_1 = 0
        \end{aligned}
    \]
    It only has a solution if $p = 1$, $q = 2$, $i_1 = 2$ and $\{ j_1, j_2 \} = \{ 1, 3 \}$. Fixing $(i_1,j_1,j_2) = (2,1,3)$ gives the solutions $(l,l,l)$ for $l\in \bbN$.

    The second possible labeling is
    \[
        \begin{tikzpicture}
            \node (n1) at (0,0) {$\dots$};
            \node (n2) at (1,0) {$j_2$};
            \node (n3) at (2,0) {$i_2$};
            \node (n4) at (3,0) {$j_1$};
            \node (n5) at (4,0) {$i_1$};
            \draw (n1) -- (n2) -- (n3) -- (n4);
            \drawReverseDoubleEdge{n4}{n5}
        \end{tikzpicture}
    \]
    (P3) and (P4) give the following system of linear equations:
    \[
        \left.
        \begin{aligned}
            3a_1 - 2b_1 &= 0 \\
            b_1 - a_1 - a_2 &= 0 \\
            3a_2 - b_1 - b_2 &= 0
        \end{aligned}
        \quad
        \right\rbrace
        \quad
        \begin{aligned}
            b_1 &= \tfrac32 a_1 \\
            a_2 &= b_1 - a_1 = \tfrac12 a_1 \\
            b_2 &= 3a_2 - b_1 = 0
        \end{aligned}
    \]
    It only has a solution if $p = 2$, $q = 1$, $\{i_1, i_2\} = \{1, 3\}$ and $j_1 = 2$. Fixing $(i_1,i_2,j_1) = (1,3,2)$ gives the solutions $(l,l,3l)$ for $l\in \bbN$.

    \emph{Type $D_n, n \ge 4$}. First, consider the labeling
    \[
        \begin{tikzpicture}
            \node (n1) at (0,0) {$\dots$};
            \node (n2) at (1,0) {$j_4$};
            \node (n3) at (2,0) {$i_2$};
            \node (n4) at (3,0) {$j_3$};
            \node (n5) at (4,0) {$i_1$};
            \node (n6) at (5,1) {$j_1$};
            \node (n7) at (5,-1) {$j_2$};
            \draw (n1) -- (n2) -- (n3) -- (n4) -- (n5);
            \draw (n5) -- (n6);
            \draw (n5) -- (n7);
        \end{tikzpicture}
    \]
    (P3) and (P4) give the following system of linear equations:
    \[
        \left.
        \begin{aligned}
            b_1 - a_1 &= 0 \\
            b_2 - a_1 &= 0 \\
            3a_1 - b_1 - b_2 - b_3 &= 0 \\
            b_3 - a_1 - a_2 &= 0
        \end{aligned}
        \quad
        \right\rbrace
        \quad
        \begin{aligned}
            b_1 &= a_1 \\
            b_2 &= a_1 \\
            b_3 &= 3a_1 - b_1 - b_2 = a_1 \\
            a_2 &= a_1 - b_3 = 0
        \end{aligned}
    \]
    It only has a solution if $p = 1$, $q = 3$, $i_1 = 2$ and $\{j_1,j_2,j_3\} = \{1,3,4\}$.
    Fixing $(i_1,j_1,j_2,j_3) = (2,1,3,4)$ gives the solutions $(l,l,l,l)$ for $l\in \bbN$.

    The second possible labeling is
    \[
        \begin{tikzpicture}
            \node (n1) at (0,0) {$\dots$};
            \node (n2) at (1,0) {$i_4$};
            \node (n3) at (2,0) {$j_2$};
            \node (n4) at (3,0) {$i_3$};
            \node (n5) at (4,0) {$j_1$};
            \node (n6) at (5,1) {$i_1$};
            \node (n7) at (5,-1) {$i_2$};
            \draw (n1) -- (n2) -- (n3) -- (n4) -- (n5);
            \draw (n5) -- (n6);
            \draw (n5) -- (n7);
        \end{tikzpicture}
    \]
    (P3) and (P4) give the following system of linear equations:
    \[
        \left.
        \begin{aligned}
            3a_1 - b_1 &= 0 \\
            3a_2 - b_1 &= 0 \\
            b_1 - a_1 - a_2 - a_3 &= 0 \\
            3a_3 - b_1 - b_2 &= 0 \\
        \end{aligned}
        \quad
        \right\rbrace
        \quad
        \begin{aligned}
            b_1 &= 3a_1 \\
            a_2 &= \tfrac13 b_1 = a_1 \\
            a_3 &= b_1 - a_1 - a_2 = a_1 \\
            b_2 &= 3a_3 - b_1 = 3a_1 - 3a_1 = 0
        \end{aligned}
    \]
    It only has a solution if $p = 3$, $q = 1$, $\{i_1,i_3,i_4\} = \{1,3,4\}$ and $j_1 = 2$. Fixing $(i_1,i_2,i_3,j_1) = (1,3,4,2)$ gives the solutions $(l,l,l,3l)$ for $l\in \bbN$.

    \emph{Type $E_6, E_7, E_8$}. By looking at the Dynkin diagram we see that
    \[
        \begin{tikzpicture}
            \node (n1) at (0,0) {$j_2$};
            \node (n2) at (1,0) {$i_2$};
            \node (n3) at (2,0) {$j_1$};
            \node (n4) at (2,1) {$i_1$};
            \node (n5) at (3,0) {$i_3$};
            \node (n6) at (4,0) {$j_3$};
            \node (n7) at (5,0) {$(i_4)$};
            \node (n8) at (6,0) {$(j_4)$};
            \draw (n1) -- (n2) -- (n3) -- (n5) -- (n6) -- (n7) -- (n8);
            \draw (n3) -- (n4);
        \end{tikzpicture}
    \]
    and
    \[
        \begin{tikzpicture}
            \node (n1) at (0,0) {$i_2$};
            \node (n2) at (1,0) {$j_2$};
            \node (n3) at (2,0) {$i_1$};
            \node (n4) at (2,1) {$j_1$};
            \node (n5) at (3,0) {$j_3$};
            \node (n6) at (4,0) {$i_3$};
            \node (n7) at (5,0) {$(j_4)$};
            \node (n8) at (6,0) {$(i_4)$};
            \draw (n1) -- (n2) -- (n3) -- (n5) -- (n6) -- (n7) -- (n8);
            \draw (n3) -- (n4);
        \end{tikzpicture}
    \]
    are the possible labelings. However, starting at $i_1$ for the first and $j_1$ for the second, we see that the computations are the same as in $D_4$. Therefore (P1) - (P4) cannot be satisfied simultaneously.

    \emph{Type $F_4$}. By looking at the Dynkin diagram we see that
    \[
        \begin{tikzpicture}
            \node (n1) at (0,0) {$i_1$};
            \node (n2) at (1,0) {$j_1$};
            \node (n3) at (2,0) {$i_2$};
            \node (n4) at (3,0) {$j_2$};
            \draw (n1) -- (n2);
            \drawDoubleEdge{n2}{n3}
            \draw (n3) -- (n4);
        \end{tikzpicture}
    \]
    and
    \[
        \begin{tikzpicture}
            \node (n1) at (0,0) {$j_1$};
            \node (n2) at (1,0) {$i_1$};
            \node (n3) at (2,0) {$j_2$};
            \node (n4) at (3,0) {$i_2$};
            \draw (n1) -- (n2);
            \drawDoubleEdge{n2}{n3}
            \draw (n3) -- (n4);
        \end{tikzpicture}
    \]
    are the possible labelings.
    However starting at $i_1$ for the first and at $j_1$ for the second gives the second $B_3$ and first $B_3$ case, respectively. Therefore (P1) - (P4) cannot be satisfied simultaneously.

    \emph{Type $G_2$}. The two possible labelings are
    \[
        \begin{tikzpicture}[baseline=(n1.base)]
            \node (n1) at (0,0) {$i_1$};
            \node (n2) at (1,0) {$j_1$};
            \drawTripleEdge{n1}{n2}
        \end{tikzpicture}
        \quad\text{and}\quad
        \begin{tikzpicture}[baseline=(n1.base)]
            \node (n1) at (0,0) {$j_1$};
            \node (n2) at (1,0) {$i_1$};
            \drawTripleEdge{n1}{n2}
        \end{tikzpicture}
    \]
    In the first case fixing $(i_1, j_1) = (1,2)$ gives the solutions $(l,l)$ for $l\in \bbN$ and in the second fixing $(i_1, j_1) = (2,1)$ gives the solutions $(l, 3l)$ for $l\in \bbN$.
\end{proof}

\begin{remark}
    We just considered the case where $\frakg$ is a finite-dimensional Lie algebra. However, it is a natural question to ask if there are monomials arising from (infinite-dimensional) symmetrizable Kac-Moody Lie algebras. Fix $k \in \{1,\cdots, p\}$ and $l \in \{1, \dots, q\}$ such that $\langle h_{i_k}, \alpha_{j_l}\rangle \langle h_{j_l}, \alpha_{i_k}\rangle > 3$ (such a pair must exist by (P1) and the assumption that $\frakg$ is infinite-dimensional). However, combining properties (P3) and (P4) we obtain
    \begin{align*}
    0 = 3a_k + \sum_{t=1}^q \langle h_{i_k}, \alpha_{j_t}\rangle b_t
    &\le 3a_k + \langle h_{i_k}, \alpha_{j_l}\rangle b_l \\
    &\le 3a_k - \sum_{s=1}^p\langle h_{i_k}, \alpha_{j_l}\rangle \langle h_{j_l}, \alpha_{i_s}\rangle a_s \\
    &\le 3a_k - \langle h_{i_k}, \alpha_{j_l}\rangle \langle h_{j_l}, \alpha_{i_k}\rangle a_k.
    \end{align*}
    This is a contradiction to our choice of $k$ and $l$. In particular, the above inequality only has a positive solution if $\frakg$ is finite-dimensional.
\end{remark}

In the following fix $(\frakg, \theta, \xi)$ to be one of the ten cases in Theorem \ref{theorem:classification}.

\begin{lemma}
    \label{lemma:theta-in-canonical-basis}
    We have $\theta^{[l]} \in \calB(\infty)$ for all $l \in \bbN$.
\end{lemma}
\begin{proof}
    This is an immediate consequence of Proposition \ref{proposition:family-of-monomials-in-canonical-basis} in view of Example \ref{example:monomials-in-canonical-basis}.
\end{proof}

\begin{theorem}
    \label{theorem:bilinear-form-values}
    We have
    \begin{enumerate}
        \item $(\theta^{[l]}, \theta^{[l]}) \in 1 + qA$,
        \item $(\theta^{[2l]}, \xi^{[l]}) \in 1 + qA$, and
        \item $(\xi^{[l]}, \xi^{[l]}) \in (l+1) + qA$.
    \end{enumerate}
\end{theorem}

\begin{remark}
    In particular Theorem \ref{theorem:bilinear-form-values} states that $\xi^{[l]}$ is not an element of the canonical basis for $l \ge 1$.
    This shows Theorem 2.1 of both \cite{Wang:TightMonomials-G2-A3} and \cite{Wang:TightMonomials-Rank2-KacMoodyLieAlgebras} does not hold in the generality stated there.
\end{remark}

\begin{proof}
    (1) is a consequence of Lemma \ref{lemma:theta-in-canonical-basis}.
    Let $\theta$ and $\xi$ be as in (\ref{equation:notation-theta}) and (\ref{equation:notation-xi}); we begin with the proof of (2).
    For this, rewrite the bilinear form into
    \begin{align*}
        (\theta^{[2l]}, \xi^{[l]}) ={}& [2a_1]_{i_1}!^{-2}\cdots [2a_p]_{i_p}!^{-2} [2b_1]_{j_1}!^{-1}\cdots [2b_q]_{j_q}!^{-1} \\
        &\times (1, e_{i_1}'^{2a_1}\cdots e_{i_p}'^{2a_p} e_{j_1}'^{2b_1}\cdots e_{j_q}'^{2b_q} e_{i_1}'^{2a_1}\cdots e_{i_p}'^{2a_p}\xi^{[l]})
    \end{align*}
    and by expanding the action of the $e_i'$ this becomes
    \begin{equation}
        \tag{I}
        \sum_{k_1=0}^{a_1} \sum_{l_1=a_1-k_1}^{2a_1-k_1} \dots \sum_{k_p=0}^{a_p} \sum_{l_p=a_p-k_p}^{2a_p-k_p} A(\ulk, \ull) B(\ulk, \ull),
    \end{equation}
    where $\ulk = (k_1, \dots, k_p)$, $\ull = (l_1, \dots, l_p)$,
    \begin{align*}
        & A(\ulk, \ull) = \\
        & \prod_{s=1}^p [2a_s]_{i_s}!^{-2} q_{i_s}^{-4a_s^2+3k_sa_s-k_s(k_s-1)/2} \qbinom{2a_s}{k_s}_{i_s} q_{i_s}^{-(2a_s-k_s)\sum_{t=1}^q \langle h_{i_s}, \alpha_{j_t}\rangle b_t} \\
        &\times q_{i_s}^{-2(2a_s-k_s)2a_s + l_s(4a_s-k_s) - l_s(l_s-1)/2} \qbinom{2a_s - k_s}{l_s}_{i_s} q_{i_s}^{-(2a_s-k_s-l_s) \sum_{t=1}^q \langle h_{i_s}, \alpha_{j_t}\rangle b_t}\\
        &\times q_{i_s}^{-2(2a_s-k_s-l_s)a_s + (2a_s-k_s-l_s)(3a_s-k_s-l_s) - (2a_s-k_s-l_s)(2a_s-k_s-l_s-1)/2} \\
        &\times q_{i_s}^{-4a_s(a_s-k_s) + (a_s-k_s)(3a_s-k_s) - (a_s-k_s)(a_s-k_s-1)/2} \qbinom{2a_s}{a_s-k_s}_{i_s} \\
        &\times q_{i_s}^{-2(a_s+k_s)(2a_s-l_s) + (2a_s-l_s)(3a_s+k_s-l_s) - (2a_s-l_s)(2a_s-l_s-1)/2} \qbinom{a_s+k_s}{2a_s-l_s}_{i_s} \\
        &\times q_{i_s}^{-2(-a_s+k_s+l_s)^2 + 2(-a_s+k_s+l_s)^2 - (k_s+l_s-a_s)(k_s+l_s-a_s-1)/2},
    \end{align*}
    and
    \begin{align*}
        & B(\ulk, \ull) = \\
        & \prod_{t=1}^q [2b_t]_{j_t}!^{-1} q_{j_t}^{-2b_t \sum_{s=1}^p \langle h_{j_t}, \alpha_{i_s}\rangle (a_s-k_s)} q_{j_t}^{-4b_t^2 + 3b_t^2 - b_t(b_t-1)/2} \qbinom{2b_t}{b_t}_{j_t} \\
        &\times q_{j_t}^{-b_t \sum_{s=1}^p\langle h_{j_t}, \alpha_{i_s}\rangle (2a_s-l_s)} q_{j_t}^{-2b_t^2 + 2b_t^2 - b_t(b_t-1)/2}.
    \end{align*}
    Let us explain the summation ranges in (I) for $s\in [1, p]$.
    The upper bound on $k_s$ is due to the fact the first $f_{i_s}$ in $\xi$ has exponent $a_s$, and the lower bound of $l_s$ comes from the requirement that for the last $f_{i_s}$ we have
    \[
        e_{i_s}'^{2a_s-k_s-l_s}f_{i_s}^{(a_s)} \neq 0,
    \]
    that is $2a_s-k_s-l_s \le a_s$.

    For every factor of the product $\prod_{s=1}^p$ in $A(\ulk, \ull)$ we have
    \begin{align*}
        & q_{i_s}^{-k_s(k_s-1)/2 - l_s(l_s-1)/2 - (2a_s-k_s-l_s)(2a_s-k_s-l_s-1)/2 - (a_s-k_s)(a_s-k_s-1)/2 } \\
        &\times q_{i_s}^{-(2a_s-l_s)(2a_s-l_s-1)/2 - (k_s+l_s-a_s)(k_s+l_s-a_s-1)/2} \\
        &\times [2a_s]_{i_s}!^{-2} \qbinom{2a_s}{k_s}_{i_s} \qbinom{2a_s - k_s}{l_s}_{i_s} \qbinom{2a_s}{a_s-k_s}_{i_s} \qbinom{a_s+k_s}{2a_s-l_s}_{i_s} \in 1 + qA,
    \end{align*}
    and for every factor of $\prod_{t=1}^q$ in $B(\ulk, \ull)$ we have
    \[
        q_{j_t}^{-b_t(b_t-1)} [2b_t]_{j_t}!^{-1} \qbinom{2b_t}{b_t}_{j_t} \in 1 + qA.
    \]
    Now, from $B(\ulk, \ull)$ we extract
    \[
        \prod_{t=1}^q q_{j_t}^{-b_t \sum_{s=1}^p \langle h_{j_t}, \alpha_{i_s}\rangle (4a_s-2k_s-l_s)} = \prod_{s=1}^p q^{-(4a_s-2k_s-l_s)\sum_{t=1}^q (\alpha_{i_s}, \alpha_{i_t}) b_t}.
    \]
    using the identity $q_{j_t}^{\langle h_{j_t}, \alpha_{i_s}\rangle} = q^{(\alpha_{i_s}, \alpha_{i_t})}$.
    Combining this with the remaining terms in $A(\ulk, \ull)$ we get
    \begin{align*}
       \prod_{s=1}^p q^{\tfrac12(\alpha_{i_s}, \alpha_{i_s}) (-9a_s^2 + 2a_sk_s - 2a_sl_s + 2k_s^2 + 2k_sl_s + 2l_s^2) - (8a_s-4k_s-2l_s) \sum_{t=1}^q (\alpha_{i_s}, \alpha_{i_t}) b_t}.
    \end{align*}
    Using the inner product version of (P3), $3(\alpha_{i_s}, \alpha_{i_s})a_s = -2\sum_{t=1}^q (\alpha_{i_s}, \alpha_{j_t}) b_t$, the above expression becomes
    \[
        q^{\tfrac12(\alpha_{i_s}, \alpha_{i_s})(15a_s^2 - 10a_sk_s - 8a_sl_s + 2k_s^2 + 2k_sl_s + 2l_s^2)}.
    \]
    Accounting for the remaining $\prod_{t=1}^q q_{j_t}^{-b_t^2}$ from $B(\ulk, \ull)$, one sees that (I) is equal to
    \[
        q^{\frac32\sum_{s=1}^p (\alpha_{i_s}, \alpha_{i_s}) a_s^2 - \frac12 \sum_{t=1}^q (\alpha_{j_t}, \alpha_{j_t}) b_t^2} \sum_{k_1=0}^{a_1} \sum_{l_1=a_1-k_1}^{2a_1-k_1} \dots \sum_{k_p=0}^{a_p} \sum_{l_p=a_p-k_p}^{2a_p-k_p}  q^{Q_1(\ulk, \ull)},
    \]
    where $Q_1(\ulk, \ull)$ is given by
    \[
        \sum_{s=1}^p \tfrac12(\alpha_{i_s}, \alpha_{i_s}) ((k_s-a_s)^2 + (l_s-a_s)^2 + (k_s+l_s-2a_s)^2 + 2a_s(3a_s-2k_s-l_s)).
    \]
    Using the inner product version of (P3) again and the one of (P4), $(\alpha_{j_t}, \alpha_{j_t})b_t = -2\sum_{s=1}^q (\alpha_{j_t}, \alpha_{i_s})a_s$, we see that
    \[
        q^{\frac32\sum_{s=1}^p (\alpha_{i_s}, \alpha_{i_s}) a_s^2 - \frac12 \sum_{t=1}^q (\alpha_{j_t}, \alpha_{j_t}) b_t^2} = 1.
    \]
    It remains to consider $Q_1(\ulk, \ull)$. Looking at the summation ranges of (I) we see that $Q_1(\ulk, \ull)$ is nonnegative, and it is zero if and only if $k_s = l_s = a_s$. This shows that every summand in (I) is in $qA$ except for the case where $k_s = l_s = a_s$ for all $1 \le s \le p$; it follows that $(\theta^{[2l]}, \xi^{[l]}) \in 1 + qA$.

    The proof of (3) is similar to the one of (2). We begin again by expanding
    \begin{align*}
        (\xi^{[l]}, \xi^{[l]}) ={}& [a_1]_{i_1}!^{-2}\cdots [a_p]_{i_p}!^{-2} [b_1]_{j_1}!^{-2}\cdots [b_q]_{j_q}!^{-2} [2a_1]_{i_1}!^{-1}\cdots [2a_p]_{i_p}!^{-1} \\
        &\times (1, e_{i_1}'^{a_1}\cdots e_{i_p}'^{a_p} e_{j_1}'^{b_1}\cdots e_{j_q}'^{b_q} e_{i_1}'^{2a_1}\cdots e_{i_p}'^{2a_p} e_{j_1}'^{b_1}\cdots e_{j_q}'^{b_q} e_{i_1}'^{a_1}\cdots e_{i_p}'^{a_p} \xi)
    \end{align*}
    into
    \begin{align*}
        \tag{II}
        & \sum_{k_1=0}^{a_1}\sum_{l_1=0}^{a_1-k_1} \sum_{k_1'=0}^{a_1-k_1} \sum_{l_1'=2a_1'-k_1-l_1-k_1'}^{2a_1-k_1'} \dots \sum_{k_p=0}^{a_p}\sum_{l_p=0}^{a_p-k_p} \sum_{k_p'=0}^{a_p-k_p} \sum_{l_p'=2a_p'-k_p-l_p-k_p'}^{2a_p-k_p'} \\
        & \sum_{m_1=0}^{b_1} \dots \sum_{m_q=0}^{b_q} A(\ulk, \ull, \ulk', \ull', \ulm) B(\ulk, \ull, \ulk', \ull', \ulm),
    \end{align*}
    where $\ulk = (k_1, \dots, k_p)$, $\ull = (l_1, \dots, l_p)$, etc.,
    \begin{align*}
        & A(\ulk, \ull, \ulk', \ull', \ulm) = \\
        &\prod_{s=1}^p [a_s]_{i_s}!^{-2} [2a_s]_{i_s}!^{-1} q_{i_s}^{-2a_s^2+2k_sa_s-k_s(k_s-1)/2} \qbinom{a_s}{k_s}_{i_s} q_{i_s}^{-(a_s-k_s)\sum_{t=1}^q \langle h_{i_s}, \alpha_{j_t}\rangle b_t} \\
        &\times q_{i_s}^{-2(a_s-k_s)2a_s + l_s(3a_s - k_s) - l_s(l_s-1)/2} \qbinom{a_s-k_s}{l_s}_{i_s} q_{i_s}^{-(a_s-k_s-l_s)\sum_{t=1}^q \langle h_{i_s}, \alpha_{j_t}\rangle b_t} \\
        &\times q_{i_s}^{-2(a_s-k_s-l_s)a_s + (a_s-k_s-l_s)(2a_s-k_s-l_s) - (a_s-k_s-l_s)(a_s-k_s-l_s-1)/2} \\
        &\times q_{i_s}^{-2(a_s-k_s)2a_s + k_s'(3a_s-k_s) + k_s'(k_s'-1)/2} \qbinom{2a_s}{k_s'}_{i_s} q_{i_s}^{-(2a_s-k_s') \sum_{t=1}^q \langle h_{i_s}, \alpha_{j_t}\rangle (b_t - m_t)} \\
        &\times q_{i_s}^{-2(2a_s-k_s')(2a_s-l_s) + l_s'(4a_s-l_s-k_s') - l_s'(l_s'-1)/2} \qbinom{2a_s}{l_s'}_{i_s} q_{i_s}^{-(2a_s-k_s'-l_s') \sum_{t=1}^q \langle h_{i_s}, \alpha_{j_t}\rangle m_t} \\
        &\times q_{i_s}^{-2(2a_s-k_s'-l_s')(k_s+l_s) + (2a_s-k_s'-l_s')(2a_s+k_s+l_s-k_s'-l_s') - (2a_s-k_s'-l_s')(2a_s-k_s'-l_s'-1)/2} \\
        &\times q_{i_s}^{-2a_s(a_s-k_s-k_s') + (a_s-k_s-k_s')(2a_s-k_s-k_s') - (a_s-k_s-k_s')(a_s-k_s-k_s'-1)/2} \qbinom{a_s}{a_s-k_s-k_s'}_{i_s} \\
        &\times q_{i_s}^{-2(k_s+k_s')(2a_s-l_s-l_s') + (2a_s-l_s-l_s')(2a_s+k_s-l_s+k_s'-l_s') - (2a_s-l_s-l_s')(2a_s-l_s-l_s'-1)/2} \\
        &\quad \times\qbinom{k_s+k_s'}{-2a_s+k_s+l_s+k_s'+l_s'}_{i_s} \\
        &\times q_{i_s}^{-2(-2a_s+k_s+l_s+k_s'+l_s')^2 + 2(-2a_s+k_s+l_s+k_s'+l_s')^2} \\
        &\times q_{i_s}^{-(-2a_s+k_s+l_s+k_s'+l_s')(-2a_s+k_s+l_s+k_s'+l_s'-1)/2},
    \end{align*}
    and
    \begin{align*}
        & B(\ulk, \ull, \ulk', \ull', \ulm) = \\
        & \prod_{t=1}^q [b_t]_{j_t}!^{-2} q_{j_t}^{-b_t\sum_{s=1}^p \langle h_{j_t}, \alpha_{i_s}\rangle (a_s-k_s)} q_{j_t}^{-2b_t^2 + 2m_tb_t - m_t(m_t-1)/2} \qbinom{b_t}{m_t}_{j_t} \\
        &\times q_{j_t}^{-(b_t-m_t)\sum_{s=1}^p \langle h_{j_t}, \alpha_{i_s}\rangle (2a_s-l_s)} q_{j_t}^{-2(b_t-m_t)b_t + (b_t-m_t)(2b_t-m_t) - (b_t-m_t)(b_t-m_t-1)/2} \\
        &\times q_{j_t}^{-b_t\sum_{s=1}^p \langle h_{j_t}, \alpha_{i_s}\rangle (a_s-k_s-k_s')} q_{j_t}^{-2(b_t-m_t)b_t + (b_t-m_t)(2b_t-m_t) + (b_t-m_t)(b_t-m_t-1)/2} \\
        &\quad \times\qbinom{b_t}{b_t-m_t}_{j_t} \\
        &\times q_{j_t}^{-m_t\sum_{s=1}^p \langle h_{j_t}, \alpha_{i_s}\rangle (2a_s-l_s-l_s')} q_{j_t}^{-2m_t^2 + 2m_t^2 - m_t(m_t-1)/2}.
    \end{align*}
    Taking care of the q-integers and q-binomials as above and rearranging the remaining terms in $A(\ulk, \ull, \ulk', \ull', \ulm)$ and $B(\ulk, \ull, \ulk', \ull', \ulm)$, we can write (II) as
    \begin{align*}
        & q^{3\sum_{s=1}^p(\alpha_{i_s}, \alpha_{i_s})a_t^2 - \sum_{t=1}^q (\alpha_{i_t}, \alpha_{i_t})b_t^2}  \sum_{k_1=0}^{a_1} \sum_{l_1=0}^{a_1-k_1} \sum_{k_1'=0}^{a_1-k_1} \sum_{l_1'=2a_1'-k_1-l_1-k_1'}^{2a_1-k_1'} \dots \\
        & \sum_{k_p=0}^{a_p}\sum_{l_p=0}^{a_p-k_p} \sum_{k_p'=0}^{a_p - k_p} \sum_{l_p'=2a_p'-k_p-l_p-k_p'}^{2a_p-k_p'} \sum_{m_1=0}^{b_1} \dots \sum_{m_q=0}^{b_q} q^{Q_2(\ulk, \ull, \ulk', \ull', \ulm)},
    \end{align*}
    where $Q_2(\ulk, \ull, \ulk', \ull', \ulm)$ is given by
    \begin{align*}
        & \sum_{s=1}^p \tfrac12(\alpha_{i_s}, \alpha_{i_s})((k_s + l_s - a_s)^2 + (k_s' + \tfrac12 l_s' - a_s)^2 \\
        &+ (k_s + k_s' - a_s)^2 + (l_s + \tfrac12 l_s' - a_s)^2 + (2a_s - l_s')(2a_s - 2k_s - k_s' - l_s)) \\
        &+ \sum_{1\le t \le q, c_{s,t} \neq 0} c_{s,t} (m_t + c_{s,t}^{-1} (\alpha_{i_s}, \alpha_{j_t}) l_s')^2,
    \end{align*}
    and the coefficients $c_{s,t}\in \bbQ_{\ge0}$ are uniquely determined by the equations
    \begin{align*}
        c_{s,t} &= 0 & \text{if } (\alpha_{i_s}, \alpha_{j_t}) = 0, \\
        \tfrac34 (\alpha_{i_s}, \alpha_{i_s}) &= \sum_{1\le t \le q, (\alpha_{i_s}, \alpha_{j_t}) \neq 0} c_{s,t}^{-1} (\alpha_{i_s}, \alpha_{j_t})^2 & \text{for } 1 \le s \le p, \\
       (\alpha_{j_t}, \alpha_{j_t}) &= \sum_{s=1}^p c_{s,t} & \text{for } 1 \le t \le q.
    \end{align*}
    Note that $c_{s,t} = 0$ if and only if $(\alpha_{i_s}, \alpha_{j_t}) = 0$. By our previous arguments, it is clear that
    \[
        q^{3\sum_{s=1}^p(\alpha_{i_s}, \alpha_{i_s})a_t^2 - \sum_{t=1}^q (\alpha_{i_t}, \alpha_{i_t})b_t^2} = 1.
    \]
    Thus, we only have to consider the zeros of $Q_2(\ulk, \ull, \ulk', \ull', \ulm)$, which occur for $l_s  = a_s - k_s$, $k_s' = a_s - k_s$, $l_s' = 2k_s$, $m_t = -2k_s c_{s,t}^{-1}(\alpha_{i_s}, \alpha_{i_t})$. By (P1) the sets $\{i_1, \dots, i_p\}$ and $\{ j_1, \dots, j_q\}$ form a partition of $I$, especially since $I$ is connected, the solution is already uniquely determined by fixing $k_1$. Finally, a case-by-case analysis shows that $(\xi^{[l]}, \xi^{[l]}) \in (l+1) + qA$.
\end{proof}

\begin{example}
    Let $\frakg$ be of type $G_2$.
    We consider (9) of Theorem \ref{theorem:classification}:
    \[
        \theta^{[a]} = f_1^{(a)} f_2^{(a)} f_1^{(a)} \text{ and } \xi^{[a]} = f_1^{(a)} f_2^{(a)} f_1^{(2a)} f_2^{(a)} f_1^{(a)}.
    \]
    We want to explicitly determine the set of zeros $S_a$ of
    \[
        Q_2(\ulk,\ull,\ulk,\ull',\ulm) = Q_2(k,l,k',l',m)
    \]
    from the previous proof on the domain defined by $0 \le k \le a$, $0 \le k' \le a-k$, $0 \le l \le a-k$, $2a-k-l-k' \le l' \le 2a-k'$ and $0 \le m \le a$.
    For this we understand $Q_2(k,l,k',l',m)$ as an element of the polynomial ring $\bbQ[k,l,k',l',m]$ with the lexicographic order $k > l > k' > l' > m$.
    In this case the only coefficient is $c_{1,2} = 6$ and thus
    \begin{align*}
        Q_2(k,l,k',l',m) ={}& (k + k' - a)^2 + (k' + \tfrac12 l' - a)^2 + (k + l - a)^2 \\
        &+ (l + \tfrac12 l' - a)^2 + (2a - l')(2a - 2k - k' - l) + 6(m - \tfrac12 l')^2.
    \end{align*}
    We now consider the ideal $J$ generated by
    \[
        k+k'-a,\quad k'+\tfrac12 l'-a,\quad k+l - a,\quad l + \tfrac12 l' - a,\quad m - \tfrac12 l'.
    \]
    Computing a Groebner basis of $J$ with respect to the above lexicographic order and applying the division algorithm to $Q_2(k,l,k',l',m)$ yields
    \begin{align*}
        Q_2(k,l,k',l',m) ={}& (k - m)(2k + 2k' + 2l + 2l' + 2m - 8a) \\
        &+ (l + m - a)(2l + 2l' - 4a) \\
        &+ (k' + m - a)(2k' + 2l' - 4a) \\
        &+ (l' - 2m)(2l' - 4m).
    \end{align*}
    Therefore, $S_1 = \{x, y\}$, where $x = (0,1,1,0,0)$ and $y = (1,0,0,2,1)$, and for $a > 1$ we have $S_a = (S_{a-1} + x) \cup (S_{a-1} + y)$.
    In particular, all zeros lie on the lattice spanned by $x$ and $y$.
\end{example}

\begin{remark}
    \label{remark:xi-canonical-basis-expansion}
    In the cases (1), (2), (7) and (8) of Theorem \ref{theorem:classification}, Baumann \cite{Baumann:CanonicalBasis-QuantumFrobeniusMorphism} found an explicit description of $\xi^{[l]}$ in terms of the canonical basis.
    More precisely, he showed that
    \[
        \xi^{[l]} = G(b_{l,1}) + \ldots + G(b_{l,l+1})
    \]
    for certain $b_{l,1}, \dots, b_{l,l+1} \in B(\infty)$.
    This agrees with Theorem \ref{theorem:bilinear-form-values} (3), thus we expect to find a similar description in the remaining cases as well.
    The following corollary is a special case of this for $l = 1$.
\end{remark}

\begin{corollary}
    \label{corollary:l1-in-canonical-basis}
    We have $\xi^{[1]} - \theta^{[2]} \in \calB(\infty)$.
\end{corollary}
\begin{proof}
    It is obvious that $\xi^{[1]} - \theta^{[2]} \in U_\bbZ^-(\frakg)$ and $\overline{\xi^{[1]} - \theta^{[2]}} = \xi^{[1]} - \theta^{[2]}$.
    Furthermore, we have
    \[
        (\xi^{[1]} - \theta^{[2]}, \xi^{[1]} - \theta^{[2]}) = (\xi^{[1]}, \xi^{[1]}) - 2(\xi^{[1]}, \theta^{[2]}) + (\theta^{[2]}, \theta^{[2]}) = 1 \mod qA.
    \]
    Thus, Theorem \ref{theorem:canonical-basis-criterion} implies that $\xi^{[1]} - \theta^{[2]} \in \pm \calB(\infty)$.

    Write $\theta^{[1]} = f_\uli^{(\ula)}$ and $\xi^{[1]} = f_\ulj^{(\ulb)}$.
    Since $\ula \in L_\uli$ and $\ulb \in L_\ulj$ fulfill the conditions of Lemma \ref{lemma:lusztig-cone-reduced-decomposition}, we know that $\uli$ and $\ulj$ are reduced words.
    By Proposition \ref{proposition:lusztig-cone-string-cone}, $\ula \in S_\uli(\infty)$ and $\ulb \in S_\ulj(\infty)$.
    We may consider $\ula$ as an element of $S_\ulj(\infty)$ by extending it to the right with zeros.
    Thus, we have $\ulb < \ula$ with respect to the lexicographic order on $S_\ulj(\infty)$.
    Hence, Proposition \ref{proposition:base-change-canonical-basis-string-basis} allows us to conclude that $\xi^{[1]} - \theta^{[2]} \in \calB(\infty)$.
\end{proof}

\begin{corollary}
    We have $\theta^{[l]}, \xi^{[l]} \in U_\bbZ^-(\frakg) \cap L(\infty) \cap \overline{L(\infty)}$ for all $l \in \bbN$.
\end{corollary}
\begin{proof}
    By definition $\theta^{[l]}, \xi^{[l]} \in U_\bbZ^-(\frakg)$, $\theta^{[l]} = \overline{\theta^{[l]}}$ and $\xi^{[l]} = \overline{\xi^{[l]}}$.
    Theorem \ref{theorem:bilinear-form-values} together with Proposition 5.1.3 of \cite{Kashiwara:CrystalBases-QuantizedUniversalEnvelopingAlgebras} imply that $\theta^{[l]}, \xi^{[l]} \in L(\infty)$.
\end{proof}

\subsection{Application to tight monomial cones}
\label{section:application-tight-monomial-cones}

We show that in $A_5$, $B_3$, $C_3$, $D_4$ and $G_2$ a reduced word $\uli$ and vector $\ula$ can be chosen such that $\ula\in L_\uli$ but $f_\uli^{(\ula)}$ is not in the canonical basis.
Due to the results of Xi \cite{Xi:CanonicalBasis-B2} in $B_2$ and Marsh \cite{Marsh:TightMonomials-QuantizedEnvelopingAlgebras} in $A_4$ it is known that in these cases every monomial belonging to a tight monomial cone is in the canonical basis.
Thus, the examples we present are minimal in the sense that for each type, it is the smallest rank that such examples exist.

\begin{enumerate}
    \item Type $A_5$. For the reduced word
    \[
        \uli = (2,4,1,3,5,2,4,1,3,5,2,4,1,3,5)
    \]
    we have
    \begin{align*}
        z_1 &= (1,1,1,2,1,2,2,1,2,1,1,1,0,0,0) \text{ and } \\
        z_2 &= (0,0,1,2,1,3,3,2,4,2,3,3,1,2,1).
    \end{align*}
    Since $f_\uli^{(z_{1/2})} = \xi_{1/2}^{[1]}$ and $(\xi_{1/2}^{[1]}, \xi_{1/2}^{[1]}) \in 2 + qA$ by Theorem \ref{theorem:bilinear-form-values}, Theorem \ref{theorem:canonical-basis-criterion} implies that $f_\uli^{(z_{1/2})}$ is not in the canonical basis.
    The argument for the remaining cases is similar.
    \item Type $B_3$. For the reduced word
    \[
        \uli = (2,1,3,2,1,3,2,1,3)
    \]
    we have
    \[
        z_3 = (1,1,2,2,1,2,1,0,0) \text{ and }
        z_4 = (0,1,2,3,2,4,3,1,2).
    \]
    \item Type $C_3$. For the reduced word
    \[
        \uli = (2,1,3,2,1,3,2,1,3)
    \]
    we have
    \[
        z_5 = (1,1,1,2,1,1,1,0,0) \text{ and }
        z_6 = (0,1,1,3,2,2,3,1,1).
    \]
    \item Type $D_4$. For the reduced word
    \[
        \uli = (2,1,3,4,2,1,3,4,2,1,3,4)
    \]
    we have
    \begin{align*}
        z_7 &= (1,1,1,1,2,1,1,1,1,0,0,0) \text{ and } \\
        z_8 &= (0,1,1,1,3,2,2,2,3,1,1,1).
    \end{align*}
    \item Type $G_2$. For the reduced word
    \[
        \uli = (1,2,1,2,1,2)
    \]
    we have
    \[
        z_9 = (1,1,2,1,1,0) \text{ and }
        z_{10} = (0,1,3,2,3,1).
    \]
\end{enumerate}

\section{Quantum Frobenius morphism}
\label{section:quantum-frobenius}

We will first show that both the Frobenius morphism and its splitting are fully compatible (defined below) with the canonical basis if $\frakg$ is of type $A_1$, $A_2$, $A_3$, or $B_2$.
Then, we will present counterexamples to full compatibility in types $A_5$, $B_3$, $C_3$, $D_4$ and $G_2$.

Let $l$ be a positive integer and for $i\in I$ let $l_i$ be the smallest positive integer such that $l_i (\alpha_i, \alpha_i)/2 \in l\bbZ$.
We define a new root datum where the simple roots are $\{ \alpha_i^* := l_i\alpha_i \}$, the simple coroots are $\{ h_i^* := l_i^{-1} h_i \}$, and the weight lattice is
\[
    P^* = \{ \lambda \in P \mid \langle h_i, \lambda\rangle \in l_i\bbZ  \text{ for all } i \in I \}.
\]

\begin{remark}
    The lattice $P^*$ should not be confused with dual lattice $P^\vee$ of $P$.
\end{remark}

We write $\frakg^*$ for the associated Kac-Moody Lie algebra, and correspondingly $U_q(\frakg^*)$ for the associated quantized enveloping algebra.
By abuse of notation we denote the generators of $\frakg^*$ and $U_q(\frakg^*)$ also by $e_i$, $f_i$ ($i \in I$) and $h$ or $q^h$ ($h \in \frakh$), respectively.
Note that $\weight f_i = \alpha_i$ for $f_i \in U_q(\frakg)$, while $\weight f_i = \alpha_i^* = l_i \alpha_i$ for $f_i \in U_q(\frakg^*)$.
Furthermore, we write $q_i^* = q^{(\alpha_i^*, \alpha_i^*)/2} = q_i^{l_i^2}$, $[n]_i^* = (q_i^{*n} - q_i^{*-n})/(q_i^* - q_i^{*-1})$, etc.

Let $R$ be an integral domain.
We assume that $R$ contains a primitive $l$-th or $2l$-th root of unity if $l$ is odd, and a primitive $2l$-th root of unity if $l$ is even.
Denote this (primitive) root of unity by $\zeta$.
The ring $R$ becomes a $\bbZ[q,q^{-1}]$-algebra via $q \mapsto \zeta$.
We now have the change of rings
\[
    U_\zeta^-(\frakg) = U_\bbZ^-(\frakg) \otimes_{\bbZ[q,q^{-1}]} R
    \text{ and }
    U_\zeta^-(\frakg^*) = U_\bbZ^-(\frakg^*) \otimes_{\bbZ[q,q^{-1}]} R.
\]

Lusztig \cite{Lusztig:IntroductionToQuantumGroups} constructed two $R$-algebra homomorphisms: the \emph{Frobenius morphism}
\[
    \Fr\colon U_\zeta^-(\frakg) \to U_\zeta^-(\frakg^*),\quad \Fr(f_i^{(n)}) = f_i^{(n/l_i)},
\]
here $f_i^{(n/l_i)}$ is understood to be $0$ if $l_i$ does not divide $n$; and the \emph{Frobenius splitting}
\[
    \Fr'\colon U_\zeta^-(\frakg^*) \to U_\zeta^-(\frakg),\quad \Fr'(f_i^{(n)}) = f_i^{(l_in)}.
\]
under the following assumptions on the root datum:
\begin{enumerate}
    \item for any $i \neq j$ in $I$ such that $l_j \ge 2$, we have $l_i \ge -\langle h_i, \alpha_j\rangle + 1$;
    \item there is no sequence $i_1, \dots, i_p, i_{p+1} = i_1$ in $I$ such that $p \ge 3$ is odd and $(\alpha_{i_s}, \alpha_{i_{s+1}}) < 0$ for $s \in [1, p]$, or equivalently, the associated Dynkin diagram has no odd cycles.
\end{enumerate}

\begin{remark}
    Observe that both maps agree with the weight grading, that is, for $f_i \in U_\zeta^-(\frakg)$
    \[
        \weight f_i^{(n)} = n \alpha_i = n l_i^{-1} \alpha_i^* = \weight \Fr(f_i^{(n)}),
    \]
    and for $f_i \in U_\zeta^-(\frakg^*)$
    \[
        \weight f_i^{(n)} = n \alpha_i^* = n l_i \alpha_i = \weight \Fr'(f_i^{(n)}).
    \]
\end{remark}

It has already been noted by Lusztig \cite{Lusztig:IntroductionToQuantumGroups} that (i) is not necessary for the existence of $\Fr'$, that (ii) may replaced with the assumption that $l$ is odd, and that (ii) is always fulfilled if $\frakg$ is of finite type.
The map $\Fr$ was later also constructed by McGerty \cite{McGerty:HallAlgebras-QuantumFrobenius} using Hall algebras under the assumption that $l$ is coprime to $(\alpha_i, \alpha_i)/2$ for all $i \in I$.
In this case, he was able to prove the existence of $\Fr$ without any further restrictions on the root datum.
It is conjectured that both maps exist without any restrictions.

On the level of crystals Kashiwara \cite{Kashiwara:Similarity-CrystalBases} constructed an embedding $S_\infty\colon B(\infty) \to B(\infty)$, which has very similar properties to the Frobenius splitting.
For any $b\in B(\infty)$ and $i \in I$:
\begin{gather*}
    \label{equation:frobenius-splitting-crystal-analogue}
    S_\infty(\tf_i b) = \tf_i^{l_i} S_\infty(b),\quad S_\infty(\te_i b) = \te_i^{l_i} S_\infty(b), \\
    \epsilon_i(S_\infty(b)) = l_i \epsilon_i(b),\quad \varphi_i(S_\infty(b)) = l_i\varphi_i(b), \quad \weight(S_\infty(b)) = \weight(b).
\end{gather*}
It is thus interesting to study the compatibility of the canonical basis with the Frobenius morphism and its splitting.
This was first done by Baumann \cite{Baumann:CanonicalBasis-QuantumFrobeniusMorphism}.

Following \emph{loc.cit.}, for $b \in B(\infty)$ we say that $G(b)$ is \emph{compatible} with the Frobenius morphism if
\[
    \Fr(G(b)) =
    \begin{cases}
        G(b'), & \text{if there is } b' \in B(\infty) \text{ with } b = S_\infty(b'), \\
        0, & \text{else}.
    \end{cases}
\]
Note for the first case that $S_\infty$ is an embedding, as such there is at most one preimage and the condition makes sense.
If the above property holds for all $b \in B(\infty)$, we say that the canonical basis is \emph{fully compatible} with the Frobenius morphism.

Similarly, we say that $G(b)$ is \emph{compatible} with the Frobenius splitting if
\[
    \Fr'(G(b)) = G(S_\infty(b));
\]
and if this property holds for all $b \in B(\infty)$, we say that the canonical basis is \emph{fully compatible} with the Frobenius splitting.

\subsection{Full compatibility}
The results on full compatibility were already mentioned in \cite{Baumann:CanonicalBasis-QuantumFrobeniusMorphism} as an observation by Littelmann, however no proof was provided.

\begin{lemma}[{\cite[Lemma 35.1.5]{Lusztig:IntroductionToQuantumGroups}}]
    For $a\in \bbZ$ and $t \in \bbN$, we have the following equality in $R$:
    \[
        \qbinom{a}{t}_i^* = \qbinom{l_i a}{l_i t}_i.
    \]
\end{lemma}

\begin{proposition}
    \label{proposition:frobenius-splitting-compatibility}
    Assume that $\frakg$ is of type $A_1$, $A_2$, $A_3$, or $B_2$.
    Then, the Frobenius splitting is fully compatible with the canonical basis.
\end{proposition}
\begin{proof}
    Let $b \in B(\infty)$.
    Then, we know (as a consequence of Proposition \ref{proposition:base-change-canonical-basis-string-basis}) that $G(S_\infty(b))$ occurs with coefficient $1$ in the expansion of $\Fr'(G(b))$ in the canonical basis of $U_\zeta^-(\frakg)$.
    Thus, it is sufficient to show that $\Fr'(G(b))$ is an element of the canonical basis to prove compatibility.

    Note that $A_1$ and $A_2$ are simply-laced, that is $l_i = l_j$ for any $i, j \in I$.
    Hence, it is clear from the definition of the canonical basis in $A_1$ and $A_2$ that it is compatible the Frobenius splitting.

    We prove the $A_3$-case using the explicit description of the canonical basis from \cite{Xi:CanonicalBasis-A3}.
    $A_3$ is simply-laced, thus $l_i = l_j$ for any $i, j \in I$; we write $t$ for this integer.
    It follows that the Frobenius splitting is compatible with the monomials (a) - (h).
    The computation for the elements (i) - (n) is similar; we restrict ourselves to verifying (i).
    In this case the element in the canonical basis of $U_{\zeta}^-(\frakg^*)$ is
    \[
        u = \sum_{0 \le k\le c} \qbinom{d+e-a-b}{k}^* f_2^{(b)} f_3^{(a+b+k)} f_2^{(d)} f_1^{(c+e+f)} f_2^{(c+e)} f_3^{(c-k)},
    \]
    where $f \ge d$, $e \ge b$ and $a+b \ge d+e$.
    The image under the Frobenius splitting is
    \[
        \Fr'(u) = \sum_{0 \le k\le c} \qbinom{td+te-ta-tb}{tk} f_2^{(tb)} f_3^{(ta+tb+tk)} f_2^{(td)} f_1^{(tc+te+tf)} f_2^{(tc+te)} f_3^{(tc-tk)}.
    \]
    Using Lemma 34.1.2 of \cite{Lusztig:IntroductionToQuantumGroups}, we see that
    \[
        \Fr'(u) = \sum_{0 \le k\le tc} \qbinom{td+te-ta-tb}{k} f_2^{(tb)} f_3^{(ta+tb+k)} f_2^{(td)} f_1^{(tc+te+tf)} f_2^{(tc+te)} f_3^{(tc-k)},
    \]
    since the $q$-binomial coefficient vanishes unless $k$ is a multiple of $t$.
    Therefore, $\Fr'(u)$ is an element of the canonical basis of $U_\zeta^-(\frakg^*)$.

    We prove the $B_2$-case using the explicit description of the canonical basis from \cite{Xi:CanonicalBasis-B2}.
    If $l$ is odd, then $l_1 = l_2$ and the same arguments as for $A_3$ may be used.
    Now, assume that $l$ is even; so in $\frakg$ the root $\alpha_1$ is long, while in $\frakg^*$ the root $\alpha_1^*$ is short.
    We only consider the cases (a), (b) and (d); the other cases are similar.
    Note that $2l_1 = l_2$ and since the in equalities in (a) - (f) define a cone we may restrict ourselves to $l_1 = 1$ and $l_2 = 2$ without loss of generality.

    For (a) the element in the canonical basis of $U_\zeta^-(\frakg^*)$ is
    \[
        u = f_1^{(a)} f_2^{(b+c)} f_1^{(2b+c)} f_2^{(d)},
    \]
    where $c \ge a$ and $b \ge d$.
    The image under the Frobenius splitting is
    \[
        \Fr'(u) = f_1^{(a)} f_2^{(2b+2c)} f_1^{(2b+c)} f_2^{(2d)}.
    \]
    We need that $2b + 2c \ge a + 2b + c$ and $4b + 2c \ge 2b + 2c + 2d$ for $f_1^{(a)} f_2^{(2b+2c)} f_1^{(2b+c)} f_2^{(2d)}$ to be in the canonical basis of $U_\zeta^-(\frakg)$.
    This is case, since the inequalities are equivalent to $c \ge a$ and $b \ge d$, respectively.

    For (b) the element in the canonical basis of $U_\zeta^-(\frakg)$ is
    \[
        u = \sum_{0\le k \le 2b+c} \qbinom{c-a}{k}_1^* f_1^{(a+k)} f_2^{(b+c)} f_1^{(2b+c-k)} f_2^{(d)},
    \]
    where $c \le a$ and $b \ge d$.
    The image under the Frobenius splitting is
    \[
        \Fr'(u) = \sum_{0\le k \le 2b+c} \qbinom{c-a}{k}_1 f_1^{(a+k)} f_2^{(2b+2c)} f_1^{(2b+c-k)} f_2^{(2d)}.
    \]
    Applying the $*$-invariance of the canonical basis (Corollary \ref{corollary:canonical-basis-star-invariance}) to (c), we see that
    \[
        \sum_{0\le k \le b'+c'} \qbinom{b'-d'}{k}_1 f_1^{(d'+k)} f_2^{(2b'+c')} f_1^{(b'+c'-k)} f_2^{(a')}
    \]
    is an element of the canonical basis of $U_\zeta^-(\frakg)$ for $c' \ge a'$ and $b' \le d'$.
    Choosing $a' = 2d$, $b' = c$, $c' = 2b$, and $d' = a$, we see that $\Fr'(u)$ is an element of the canonical basis of $U_\zeta^-(\frakg)$.

    For (d) the element in the canonical basis of $U_\zeta^-(\frakg^*)$ is
    \[
        u = \sum_{\substack{0 \le j \le 2b+c \\ 0 \le k \le b +c}} \qbinom{c-a}{j}_1^* \qbinom{b-d-j}{k}_2^* f_1^{(a+j)} f_2^{(b+c-k)} f_1^{(2b+c-j)} f_2^{(d+k)}.
    \]
    where $c \le a$, $b \le d$ and $c + d \ge a + b$.
    The image under the Frobenius splitting is
    \[
        \Fr'(u) = \sum_{\substack{0 \le j \le 2b+c \\ 0 \le k \le b +c}} \qbinom{c-a}{j}_1 \qbinom{2b-2d-2j}{2k}_2 f_1^{(a+j)} f_2^{(2b+2c-2k)} f_1^{(2b+c-j)} f_2^{(2d+2k)}.
    \]
    Using the $*$-invariance on (e), we see that
    \[
        \sum_{\substack{0 \le j \le 2b'+c' \\ 0 \le k \le b'+c'}} \qbinom{c'-a'-2k}{j}_2 \qbinom{b'-d'}{k}_1 f_1^{(d'+k)} f_2^{(2b'+c'-j)} f_1^{(b'+c'-k)} f_2^{(a'+j)}
    \]
    is an element of the canonical basis of $U_\zeta^-(\frakg)$ for $c' \le a'$, $b' \le d'$ and $c' + 2d' \le a' + 2b'$.
    Choose $a' = 2d$, $b' = c$, $c' = 2b$, and $d' = a$.
    Observing that the first binomial coefficient vanishes unless $j$ is even, we see that $\Fr'(u)$ is an element of the canonical basis of $U_\zeta^-(\frakg)$.
\end{proof}

\begin{proposition}
    \label{proposition:frobenius-morphism-compatibility}
    Assume that $\frakg$ is of type $A_1$, $A_2$, $A_3$, or $B_2$.
    Then, the Frobenius morphism is fully compatible with the canonical basis.
\end{proposition}
\begin{proof}
    For an element $u$ of the canonical basis of $U_\zeta^-(\frakg)$, it is sufficient to show that $\Fr(u) = 0$ if $u$ is not in the image of the Frobenius splitting, since $\Fr \circ \Fr' = \id$.

    The statement for $A_1$ and $A_2$ is clear from the definition of the Frobenius morphism; for $A_3$ we will again consider only the case (i) from \cite{Xi:CanonicalBasis-A3}:
    \[
        u = \sum_{0 \le k\le c} \qbinom{d+e-a-b}{k} f_2^{(b)} f_3^{(a+b+k)} f_2^{(d)} f_1^{(c+e+f)} f_2^{(c+e)} f_3^{(c-k)},
    \]
    where $f \ge d$, $e \ge b$ and $a+b \ge d+e$.
    From the definition of the Frobenius morphism we see that $\Fr(u) = 0$, unless $l$ divides all exponents.
    In particular, $l$ must divide $d + (c+e) - (a+b+k) - (c-k) = d+e-a-b$.
    Hence, for a summand not to vanish, $k$ must be a multiple of $l$, as otherwise the $q$-binomial would be zero in $R$ by Lemma 34.1.2 of \cite{Lusztig:IntroductionToQuantumGroups}.
    It follows that every summand in the definition of $u$ vanishes under $\Fr$ unless $a,b,c,d,e$ and $f$ are multiples of $l$.
    In this case, the arguments in the proof for the Frobenius splitting imply that $u$ is indeed in the image of the Frobenius splitting.

    For $B_2$, we only consider the case (d) from \cite{Xi:CanonicalBasis-B2}:
    \[
        u = \sum_{\substack{0 \le j \le 2b+c \\ 0 \le k \le b +c}} \qbinom{c-a}{j}_2 \qbinom{b-d-j}{k}_1 f_2^{(a+j)} f_1^{(b+c-k)} f_2^{(2b+c-j)} f_1^{(d+k)},
    \]
    where $c \le a$, $b \le d$ and $c + d \ge a + b$.
    Any summand vanishes under the Frobenius morphism unless $l_1$ divides $b+c-k$ and $d+k$, and $l_2$ divides $a+j$ and $2b+c-j$.
    Since $l_1 \in \{ l_2/2, l_2 \}$, we must have that $l_1$ divides
    \[
        (2b+c-j) - (b+c-k) - (d+k) = b-d-j.
    \]
    Thus, any summand vanishes unless $k$ is divisible by $l_1$.
    Furthermore, we must have that $l_2 \in \{ l_1, 2l_1 \}$ divides
    \[
        2(b+c-k) - (2b+c-j) - (a+j) = c-a-2k.
    \]
    Since, $2k$ is divisble by $l_2$, in particular, $c-a$ must be divisible by $l_2$.
    Hence, any summand vanishes unless $j$ is divisble by $l_2$.
    We conclude that every summand in the definition of $u$ vanishes under $\Fr$ unless $a$ and $c$ are multiples of $l_2$, and $b$ and $d$ are multiples of $l_1$.
    Now, the arguments in the proof for the Frobenius splitting imply that $u$ is indeed in the image of the Frobenius splitting.
\end{proof}

\subsection{Counterexamples to full compatibility}

The first counterexamples to full compatibility of the Frobenius morphism and splitting with the canonical basis were provided by Baumann \cite{Baumann:CanonicalBasis-QuantumFrobeniusMorphism} for the types $A_5$ and $D_4$.
The monomials he used in the counterexamples are our monomials $\xi_1^{[l]}$, $\xi_2^{[l]}$ for $A_5$ and $\xi_7^{[l]}$, $\xi_8^{[l]}$ for $D_4$ (see Theorem \ref{theorem:classification}).
We present a different argument for $A_5$ and $D_4$; the results for $B_3$, $C_3$ and $G_2$ are new.

First, we present the counterexamples for the Frobenius splitting.
For simplicity assume that $l$ is coprime to $(\alpha_i, \alpha_i)/2$ for all $i\in I$.
Let $\frakg$ be of type $A_5$, $B_3$, $C_3$, $D_4$ or $G_2$ and let $\theta$ and $\xi$ be chosen from the list in Theorem \ref{theorem:classification} correspondingly.
Then, $\xi^{[1]} - \theta^{[2]} \in \calB(\infty)$ by Corollary \ref{corollary:l1-in-canonical-basis}.
We have
\[
    \Fr'(\xi^{[1]} - \theta^{[2]}) = \xi^{[l]} - \theta^{[2l]}
\]
and computing the bilinear form gives
\[
    (\xi^{[l]} - \theta^{[2l]}, \xi^{[l]} - \theta^{[2l]}) = (\xi^{[l]}, \xi^{[l]}) - 2(\xi^{[l]}, \theta^{[2l]}) + (\theta^{[2l]}, \theta^{[2l]}) = l.
\]
Hence, $\xi^{[l]} - \theta^{[2l]} \notin \calB(\infty)$, by Theorem \ref{theorem:canonical-basis-criterion}.

For the counterexamples for the Frobenius morphism, we need to distinguish between types.

If $\frakg$ is of type $A_5$, $D_4$ or $G_2$, we consider $l = 2$ (assuming that $l$ is coprime to $(\alpha_i, \alpha_i)/2$ for all $i\in I$).
Then,
\[
    \Fr(\xi^{[1]} - \theta^{[2]}) = -\theta^{[1]} \notin \calB(\infty).
\]
This follows directly from the fact that $\theta^{[1]} \in \calB(\infty)$ by Lemma \ref{lemma:theta-in-canonical-basis}.

If $\frakg$ is of type $C_3$, we consider $l = 2$ and the bilinear form determined by $(\alpha_1, \alpha_1) = (\alpha_2, \alpha_2) = 2$ and $(\alpha_3, \alpha_3) = 4$.
Thus, $l_1 = l_2 = 2$ and $l_3 = 1$.
We have
\[
    f_2 (f_1 f_3) f_2^{(2)} (f_3 f_1) f_2 - f_2^{(2)} (f_1^{(2)} f_3^{(2)}) f_2^{(2)} \in \calB(\infty),
\]
by Corollary \ref{corollary:l1-in-canonical-basis}, and
\[
    \Fr(f_2 (f_1 f_3) f_2^{(2)} (f_3 f_1) f_2 - f_2^{(2)} (f_1^{(2)} f_3^{(2)}) f_2^{(2)}) = -f_2 (f_1 f_3^{(2)}) f_2.
\]
Since $f_2 (f_1 f_3^{(2)}) f_2$ is an element of the canonical basis of $U_\zeta^-(\frakg^*)$ (type $B_3$) by Lemma \ref{lemma:theta-in-canonical-basis}, $-f_2 (f_1 f_3^{(2)}) f_2$ is not.

If $\frakg$ is of type $B_3$, we consider $l = 4$ and the bilinear form determined by $(\alpha_1, \alpha_1) = (\alpha_2, \alpha_2) = 4$ and $(\alpha_3, \alpha_3) = 2$.
Thus, $l_1 = l_2 = 2$ and $l_3 = 4$.
We have
\[
    f_2 (f_1 f_3^{(2)}) f_2^{(2)} (f_3^{(2)} f_1) f_2 - f_2^{(2)} (f_1^{(2)} f_3^{(4)}) f_2^{(2)} \in \calB(\infty),
\]
by Corollary \ref{corollary:l1-in-canonical-basis}, and
\[
    \Fr(f_2 (f_1 f_3^{(2)}) f_2^{(2)} (f_3^{(2)} f_1) f_2 - f_2^{(2)} (f_1^{(2)} f_3^{(4)}) f_2^{(2)}) = -f_2 (f_1 f_3) f_2.
\]
Since $f_2 (f_1 f_3) f_2$ is an element of the canonical basis of $U_\zeta^-(\frakg^*)$ (type $C_3$) by Lemma \ref{lemma:theta-in-canonical-basis}, $- f_2 (f_1 f_3) f_2$ is not.

\begin{remark}
    In view of Remark \ref{remark:xi-canonical-basis-expansion}, we expect that the counterexamples for the Frobenius morphism can be generalized to all $l \ge 2$ that are coprime to $(\alpha_i, \alpha_i)/2$ for all $i\in I$.
\end{remark}

\printbibliography

\end{document}